\documentclass[11pt,reqno,twoside]{article}

\usepackage[margin=0.2in]{geometry}
\usepackage{amsmath, amsthm, amssymb, mathrsfs}
\usepackage{hyperref}

\hypersetup{colorlinks=true, pdfstartview=FitV, linkcolor=blue, citecolor=blue, urlcolor=blue}
\usepackage{times}
\usepackage{mathtools}
\usepackage{stmaryrd}
\usepackage{bm}
\usepackage{mathabx}
\usepackage{geometry}
\usepackage{enumerate}
\usepackage{authblk}
\usepackage{verbatim}

\usepackage{graphicx,epstopdf}
\usepackage{subcaption}
\usepackage{array}
\usepackage[table]{xcolor}

\usepackage{fancyhdr}

\numberwithin{equation}{section}
\newtheorem{Theorem}{Theorem}[section]
\newtheorem{Proposition}[Theorem]{Proposition}

\newtheorem{innerassump}{Assumption}
\newenvironment{Assumption}[1]
  {\renewcommand\theinnerassump{\textbf{#1}}\innerassump}
  {\endinnerassump}

\newtheorem{Remark}[Theorem]{Remark}
\newtheorem{Lemma}[Theorem]{Lemma}

\newcommand{\PP}{\mathbb{P}}
\newcommand{\RR}{\mathbb{R}}
\newcommand{\NN}{\mathbb{N}}
\newcommand{\ZZ}{\mathbb{Z}}
\newcommand{\TT}{\mathbb{T}}
\newcommand{\bD}{\bm{D}}
\newcommand{\bU}{\bm{U}}
\newcommand{\bV}{\bm{V}}
\newcommand{\bk}{\bm{k}}
\newcommand{\bv}{\bm{v}}
\newcommand{\bu}{\bm{u}}
\newcommand{\bx}{\bm{x}}
\newcommand{\bz}{\bm{0}}
\newcommand{\be}{\bm{e}}
\newcommand{\bo}{\bm{1}}
\newcommand{\cC}{{\mathcal C}}

\newcommand{\cF}{{\mathcal F}}
\newcommand{\cK}{{\mathcal K}}
\newcommand{\cM}{{\mathcal M}}
\newcommand{\cO}{{\mathcal O}}
\newcommand{\cU}{{\mathcal U}}

\renewcommand{\P}{{\sf P}}
\renewcommand{\S}{{\sf S}}
\newcommand{\J}{{\sf J}}
\newcommand{\sH}{{\mathscr H}}

\renewcommand{\i}{{\rm i}}
\renewcommand{\r}{{\rm ref}}
\renewcommand{\l}{{\rm l}}
\newcommand{\h}{{\rm h}}

\DeclareMathOperator{\Id}{Id}
\DeclareMathOperator{\Span}{span}
\DeclareMathOperator{\Diag}{Diag}
\DeclareMathOperator{\dd}{{\rm d}\!}

\DeclarePairedDelimiter\norm{\big\lvert}{\big\rvert}

\begin{document}%%%%%%%%%%%%%%%%%%%%%%%%%%%%%%%%
		
	\title{The Fourier spectral approach to the spatial discretization of quasilinear hyperbolic systems}
	
	\author[1]{Vincent Duchêne}
	\author[2]{Johanna Ulvedal Marstrander}
	\affil[1]{IRMAR, Univ. Rennes , F-35000 Rennes, France.}
	\affil[2]{Department of Mathematical Sciences, NTNU, 7491 Trondheim, Norway}
	\date{\today}
	
	\maketitle 
	
	\begin{abstract}
		We discuss the rigorous justification of the spatial discretization by means of Fourier spectral methods of quasilinear first-order hyperbolic systems.
		We provide uniform stability estimates that grant spectral convergence of the (spatially) semi-discretized solutions towards the corresponding continuous solution provided that the underlying system satisfies some suitable structural assumptions. We consider a setting with sharp low-pass filters and a setting with smooth low-pass filters and argue that ---at least theoretically--- smooth low-pass filters are operable on a larger class of systems. While our theoretical results are supported with numerical evidence, we also pinpoint some behavior of the numerical method that currently has no theoretical explanation.
	\end{abstract}


	\tableofcontents

	\section{Introduction}\label{sec.intro}
	In this work we shall consider the spatial discretization by means of Fourier spectral methods of systems of the form
	\begin{equation}\label{eq.hyp}
		\partial_t \bU+\sum_{j=1}^d A_j(\bU)\partial_{x_j}\bU=\bz, \quad \bU\vert_{t=0} = \bU^0.
	\end{equation}
	where for all $j\in\{1,\dots,d\}$ and for all $\bU\in\RR^n$, $A_j(\bU)$ are matrices satisfying Assumption~\ref{assump.A1}.
	\begin{Assumption}{A.1}\label{assump.A1}
		For all \(j\in \{1,\dots,d\}\) and \(\bU \in \RR^{n}\), \(A_j(\bU)\) are $n$-by-$n$ real-valued matrices. 
		We assume that all entries of $A_j(\cdot)$ are polynomial.
	\end{Assumption}
	\begin{Remark}
		Assuming that the system~\eqref{eq.hyp} has only polynomial nonlinearities may seem an over-restrictive assumption. This assumption is motivated by two considerations. Firstly, the Fourier spectral method may be efficiently implemented only within this framework; see Remark~\ref{R.spectral} below. Secondly, as far as we know, the literature lacks a theory analogous to para-differential calculus (see {\em e.g.} ~\cite[Chapter 5]{Metivier08}) on Sobolev spaces of periodic functions, which would provide the composition estimates in Proposition~\ref{prop.composition_estimates} for general composition functions. All our results apply without assuming polynomial nonlinearities if Proposition~\ref{prop.composition_estimates} holds without that assumption.
	\end{Remark}
	We shall assume further additional structural assumptions on the system~\eqref{eq.hyp}, depending on the needs. In particular we shall always consider Friedrichs-symmetrizable systems, which guarantees the hyperbolicity of the system, and local-in-time well-posedness of the initial-value problem in $L^2$-based Sobolev spaces of sufficiently high regularity index, $\bU^0\in H^s((2\pi\TT)^d)^n$ with $s>d/2+1$ (see Subsection~\ref{sec.def.not}). For simplicity, we consider in this work $2\pi$-periodic functions in all spatial directions. The results extend straightforwardly to more general periodic frameworks, and analogous results in the $n$-dimensional Euclidean space could be obtained with some simple adaptations.

	Let \(\mathcal{T}_N\) be the space of trigonometric polynomials of degree \(N\):
	\[\mathcal{T}_N \coloneq \Span\{\exp(\i\bk \cdot \bx), \bk\in \ZZ^d, \left|k_j \right|\leq N, j=1,\ldots,d\}, \text{ and } \mathcal{T}_N^{n} \coloneq \underbrace{\mathcal{T}_N \times\ldots\times \mathcal{T}_N}_{n \text{ times}}.\]
	Let \(\P_N \colon L^2((2\pi\TT)^d)^n \to  \mathcal{T}_N^{n}\) be the \(L^2\)-projection operator onto \(\mathcal{T}_N^{n}\): \(\P_N = \Diag(P_N(D))\), where \(\P_N\) is a Fourier multiplier with symbol \(P_N (\cdot)= \bo_{\llbracket -N,N \rrbracket^d}(\cdot)\). We refer to \(\P_N\) as a {\bf sharp low-pass filter}. For any \(0\leq r\leq s\) and \(\bU\in H^s((2\pi\TT)^d)^n\) Sobolev space of order \(s\), \(\P_N\) satisfies (see {\em e.g.}~\cite[(5.1.10),(5.8.4)]{CanutoHussainiQuarteroniEtAl06})
	\begin{equation}\label{eq.PNest}
		\left|\bU - \P_N \bU\right|_{H^r} \leq C(d,s,r) \left\langle N \right\rangle^{r-s} \left|\bU \right|_{H^s}.
	\end{equation}
	This kind of estimates are referred to in the literature as {\em spectral convergence}, and we will follow this terminology.
	
	The standard Fourier spectral method {(see {\em e.g.}~\cite{KreissOliger72,GottliebOrszag77,MajdaMcDonoughOsher78} or~\cite[Section~3]{BardosTadmor15})} for the spatial discretization of the problem~\eqref{eq.hyp} amounts to  seeking solutions \(\bU_N \colon t\mapsto \mathcal{T}_N^{n}\) to the problem
	\begin{equation}\label{eq.hyp-sharp}
		\partial_t \bU_N+\P_N\left(\sum_{j=1}^d A_j(\bU_N)\partial_{x_j}\bU_N\right)=\bz, \quad \bU_N\vert_{t=0} = \P_N\bU^0.
	\end{equation}
	\begin{Remark}\label{R.spectral}
		Let us recall that one of the great assets of Fourier spectral methods is that spatial differentiation and multiplication can be very efficiently performed (up to machine precision) by means of Fast Fourier Transform (FFT/IFFT) and multiplication at spatial collocations points, provided suitable dealiasing operations are performed. In practice, if the entries of $A_j(\cdot)$ are all polynomial with maximal degree $p$, then one computes (following Orszag's rule~\cite{Orszag71}) $N\frac{p+2}{2}$ modes, and applying the projection $\P_N$ after multiplications at collocation points performs the necessary dealiasing. If entries of $A_j(\cdot)$ are not polynomials, then one typically uses {\em pseudo-}spectral schemes that follow the aforementioned strategy but cannot be formulated as in~\eqref{eq.hyp-sharp}.  
		
		We do not discuss in this work the full time-space discretization of~\eqref{eq.hyp}, that is well-suited numerical time integrators for~\eqref{eq.hyp-sharp}. The implementation we used for numerical experiments is described in more detail in Subsections~\ref{sec.SV-num-1} and~\ref{sec.SV-num-2}.  
	\end{Remark}

	{\bf In this work we shall discuss the long-time {({\em i.e.} uniform with respect to $N$)} existence of solutions to~\eqref{eq.hyp-sharp} and {spectral} convergence towards solutions to~\eqref{eq.hyp} as $N\to\infty$.} 
	
	Our results will depend on the structure of the system. As we shall see, the solution to~\eqref{eq.hyp-sharp} converges towards the corresponding solution to the problem~\eqref{eq.hyp} (assuming sufficient regularity) whenever the system is symmetric. If the system is only symmetrizable, the situation is more complicated. In order to deal with this situation we consider {\bf smooth low-pass filters}, \(\S_N \colon L^2((2\pi\TT)^d)^n \to \mathcal{T}_N^{n}\) where
	${\S_N=\Diag(S_N(D))}$ is a Fourier multiplier with symbol $S_N(\cdot)=S(\cdot/N)$ where $S$ is even and satisfies
	\[\begin{cases}
		S(\bk)=1&\text{if $\max_{j=1,\ldots, d}|k_j|\leq 1/2$,}\\
		S(\bk)=0&\text{if $\min_{j=1, \ldots, d}|k_j|\geq 1$,}\\
		S(\bk)\in{[}0,1{]}&\text{otherwise,}
	\end{cases}
	\]
	and $S^{1/2}$ is Lipschitz-continuous.
	 When \(d=1\), an example of such a function is $S_1(\cdot)\coloneq\max\big(0,\min\big(1,2-2|\cdot|\big)\big)^2$. When $d\geq 2$, one can set $S_d((k_1,\dots,k_d))\coloneq S_1(k_1)\times\cdots \times S_1(k_d)$.
	The advantage of such smooth low-pass filters is that ---contrarily to the sharp low-pass filter--- they satisfy commutator estimates with gains of regularity uniformly with respect to \(N\); see Proposition~\ref{prop.commutator_0}. These are crucial to parts of the analysis.
	Notice that because $| 1-S(\bk)|\leq 1=|1-P_{1/2}(\bk)|$ when $\max_{j=1,\ldots, d}|k_j|\geq 1/2$ and $1-S(\bk)=0=1-P_{1/2}(\bk)$ otherwise, we infer from~\eqref{eq.PNest} the corresponding spectral convergence estimate
		\begin{equation}\label{eq.SNest}
		\left|\bU - \S_N \bU\right|_{H^r} \leq \left|\bU - \P_{N/2} \bU\right|_{H^r} \leq C(d,s,r) \left\langle N \right\rangle^{r-s} \left|\bU \right|_{H^s}.
	\end{equation}
	The spatial discretization of~\eqref{eq.hyp} using smooth low-pass filters could amount to finding a solution \(\bU_N \colon t\mapsto \mathcal{T}_N^{n}\) to
	\begin{equation}\label{eq.hyp-smooth0}
		\partial_t \bU_N+\S_N\left(\sum_{j=1}^d A_j(\bU_N)\partial_{x_j}\bU_N\right)=\bz, \quad \bU_N\vert_{t=0} = \P_N\bU^0.
	\end{equation}
	We also consider variants of this system such as
	\begin{equation}\label{eq.hyp-smooth}
		\partial_t \bU_N+\sum_{j=1}^d (A_j^0 + \S_N (A_j^1(\bU_N)[\circ]))\partial_{x_j}\bU_N=\bz, \quad \bU_N\vert_{t=0} = \P_N \bU^0.
	\end{equation}
	where \(A_j^0=A_j(\bz)\) and \(A_j^1(\bU) = A_j(\bU)- A_j^0 \). Indeed, applying smooth low-pass filters to linear terms is unnecessary, especially when one uses exponential time integrators (see {\em e.g.} Program 27 in~\cite{Trefethen00}).
	Notice the distinction between~\eqref{eq.hyp-smooth0} and~\eqref{eq.hyp-smooth} is only necessary when using smooth low-pass filters since when using the sharp low-pass filter $\P_N$ one has \(\P_N \bU_N = \bU_N\) and $\P_N$ commutes with $A_j^0$. 
	\medskip
	
	\paragraph{Outline} Let us now describe the structure and main results of this work.
	Symmetric systems are discussed in Section~\ref{sec.sym}. We consider sharp and smooth low-pass filters in Subsection~\ref{sec.sym.sharp} and~\ref{sec.sym.smooth} respectively. We obtain convergence of the semi-discretized solutions in both cases, stated in Propositions~\ref{prop.sym-sharpdifference} and~\ref{prop.sym-smoothdifference}. Symmetrizable systems are discussed in Section~\ref{sec.symable}. Subsection~\ref{sec.symable.smooth} concerns smooth low-pass filters and we obtain analogous convergence results, stated in Proposition~\ref{prop.symable-difference}. The case with sharp low-pass filters is treated in Subsection~\ref{sec.symable.sharp} and in order to secure spectral convergence, more stringent structural assumptions on the system are required. This yields Proposition~\ref{prop.symable-sharpdifference}. Numerical experiments illustrating and investigating the sharpness of our theoretical results are provided in Section~\ref{sec.illustration}.

	\subsection{Motivation and related works}
	
	Our work was motivated by the study of Boussinesq and Whitham--Boussinesq systems that are nonlinear dispersive models for the propagation of surface gravity waves~\cite{Lannes,MM4WW}. The Fourier spectral method is especially indicated for the spatial discretization of these systems since the nonlinear contributions are quadratic (recall Remark~\ref{R.spectral}) and the dispersive contributions take the form of Fourier multipliers. The spectral convergence of discretized versions of Boussinesq models towards the corresponding continuous solutions was proved in~\cite{XavierRinconAlfaroVigoEtAl18} and~\cite{DougalisDuranSaridaki23}. These results have a shortcoming in that they lack uniformity in the non-dispersive limit. In order to clarify this point, let us consider the Benjamin--Bona--Mahony (BBM) and (a variant of) the Whitham equations, which read respectively
	\[(\Id-\mu\partial_x^2)\partial_t u +\partial_x u+u\partial_x u=0, \quad \text{ and } \quad  \partial_t u+\frac{\tanh(\sqrt\mu |D|)}{\sqrt\mu|D|}\big(\partial_x u+u\partial_x u\big)=0.\]
	The aforementioned Boussinesq (resp. Whitham--Boussinesq) systems can be loosely considered as systems extending the BBM (resp. Whitham) scalar equation, in the same way the Saint-Venant system discussed in Section~\ref{sec.illustration} extends the inviscid Burgers equation. All these equations provide valid approximations of water waves provided (among other assumptions) the shallowness parameter $\mu>0$ satisfies $\mu\ll1$; see~\cite{Lannes,MM4WW}.
	
	Results proved in~\cite{XavierRinconAlfaroVigoEtAl18} and~\cite{DougalisDuranSaridaki23} and adapted to the simplest case of scalar equations take the form
			\[\norm{u_N(t,\cdot)}_{H^s}\leq C_\mu(t) \norm{u_N\big\vert_{t=0}}_{H^s} \quad \text{ and } \quad  \norm{(u-u_N)(t,\cdot)}_{H^r}\leq C_\mu(t) N^{r-s}  \qquad \text{ for any $s>1/2$ and $ 0\leq r\leq s$,}\]
	where $u$ is the solution to the scalar equation, and $u_N$ the corresponding solution to the semi-discretized equation
		\[(\Id-\mu\partial_x^2)\partial_t u_N +\partial_x u_N+\P_N(u_N\partial_x u_N)=0, \quad \text{ and } \quad  \partial_t u_N+\frac{\tanh(\sqrt\mu |D|)}{\sqrt\mu|D|}\big(\partial_x u_N+\P_N(u_N\partial_x u_N)\big)=0.\]
	The aforementioned shortcoming is that $C_\mu(t)$ depends nonuniformly on $\mu$ in the non-dispersive limit $\mu\ll1$, typically $C_\mu(t)\lesssim 
	\exp(\mu^{-1} t)$. This is inconsistent with the standard well-posedness theory for initial data in Sobolev spaces $H^{s}(2\pi\TT)$, $s>3/2$, which holds uniformly with respect to $\mu\in (0,1]$ (see {\em e.g.}~\cite[Proposition~6]{KleinLinaresPilodEtAl18} and~\cite{SautXu12,Emerald22,Paulsen22} for the corresponding results on the Boussinesq and Whitham--Boussinesq systems). The reason for this discrepancy is that for the spectral convergence results, stability estimates on the BBM and Whitham equations are obtained by viewing them as {\em semilinear systems} which can be formulated using the Duhamel formula
	\[u(t,\cdot)=\int_0^t e^{-(t-\tau)L_\mu(D)}L_\mu(D)\big(\frac{u(\tau,\cdot)^2}2\big)\dd\tau\]
	where $L_\mu(D)=\frac{\partial_x}{\Id-\mu\partial_x^2}$ for the BBM equation, and $L_\mu(D)=\frac{\tanh(\sqrt\mu |D|)}{\sqrt\mu|D|}\partial_x$ for the Whitham equation. Notice that in both cases $L_\mu(D)\in\mathcal B(L^2(2\pi\TT);L^2(2\pi\TT))$, but that $\sup\{\norm{L_\mu v}_{L^2} \ : \ \norm{v}_{L^2}=1\}$ is not uniformly bounded with respect to $\mu\in(0,1]$, which is the source of the issue when $\mu\ll 1$. On the other hand, one can view the BBM and Whitham equation as perturbations of the inviscid Burgers equation, with skew-symmetric dispersive terms that are inconspicuous for the energy method. This leads to stability estimates in $H^{s}(2\pi\TT)$, $s>3/2$, that are uniform with respect to $\mu\ll1$. Of course obtaining such results on dispersive {\em systems} requires a good understanding of the underlying (non-dispersive) quasilinear systems, which is the focus of the current work.
	
	The rigorous analysis of semi-discretized or fully (space and time) discretized Fourier spectral schemes for semilinear equations is a very rich and active topic, which is impossible to summarize within a few lines; let us simply mention~\cite{CabreraCalvoSchratz22,CabreraCalvoRoussetSchratz22} which are particularly relevant as they specifically consider the BBM equation and pay attention to the non-dispersive limit (although together with vanishing nonlinearity, that is the long wave limit).
	In contrast, to the best of our knowledge, there are only a handful of works dedicated to the rigorous analysis of the Fourier method for {\em quasilinear systems}, culminating with the work of Bardos and Tadmor~\cite{BardosTadmor15} (following~\cite{Tadmor87,Tadmor89,GoodmanHouTadmor94}; see also the review paper~\cite{GottliebHesthaven01}). Here the authors consider the inviscid Burgers equation, as well as the one-dimensional isentropic Euler equation and the incompressible Euler equations. The first two systems belong to the class of equations we study. Specifically, the inviscid Burgers equation has the symmetric structure we employ in Subsection~\ref{sec.sym.sharp}, while the isentropic Euler equation in Lagrangian coordinates enjoys the Hamiltonian structure discussed in Subsection~\ref{sec.symable.sharp}. The incompressible Euler equations 
	\[ \partial_t \bu + \PP((\bu\cdot\nabla)\bu)=\bz, \ \PP(\bu)=\bu, \qquad \PP(\bu)\coloneq  \bu-\nabla\Delta^{-1}\nabla\cdot\bu\]
	does not belong to the class of equations studied in this work due to the presence of the Leray projection operator $\PP$ but it would not be difficult to extend our analysis (specifically concerning the symmetric situation) to this system. Let us mention that while our analysis is very similar, the estimates obtained in~\cite{BardosTadmor15} are not as sharp as the ones obtained in our work due to different choices when performing stability estimates, as the former take the form (see Theorem 3.1 therein)
	\[\norm{(\bU-\bU_N)(t,\cdot)}_{L^2}\leq C(t) N^{-s} \norm{\bU\vert_{t=0}}_{H^s}+N^{\frac34-\frac{s}2}\max_{\tau\leq t} \norm{\bU(\tau,\cdot)}_{H^s},\]
	while our results do not feature the second contribution. Quite interestingly, the authors in~\cite{BardosTadmor15} also prove the {\em emergence of spurious oscillations} of the semi-discretized solution after the critical time of shock formation for the continuous solution of the inviscid Burgers equation.  It would be interesting to study this problem when smooth low-pass filters are used and compare with the spectral viscosity method described therein.

	\subsection{Definitions and notations}
	\label{sec.def.not}

	In this section, we introduce a few notations used throughout the work. 

	Let \(L^2((2\pi\TT)^d)\) be the Lebesgue space of real-valued, square-integrable functions on the {$2\pi$-periodic} torus and \[ L^2((2\pi\TT)^d)^{n} = \underbrace{L^2((2\pi\TT)^d)\times\ldots\times L^2((2\pi\TT)^d)}_{n \text{ times}}.\]
	We endow \(L^2((2\pi\TT)^d)\) with the standard Lebesgue norm denoted $\big|\,\cdot\, \big|_{L^2}$, and the corresponding inner-product is denoted $\big(\cdot,\cdot\big)_{L^2}$. Similarly, we denote by \(L^\infty((2\pi\TT)^d)\) the Lebesgue space of bounded functions and \(W^{1,\infty}((2\pi\TT)^d)\) the space of Lipschitz continuous functions, endowed with their natural norms.
	
	We use the notation \(\left\langle \cdot \right\rangle = (1 + \left|\cdot \right|^2)^{1/2}\) and \(\Lambda^s = (1- \Delta)^{s/2}\), {\em i.e.} the Fourier multiplier with symbol \(\left\langle \cdot \right\rangle^{s}\) (see {\em e.g.}~\cite{Metivier08} for Fourier multipliers). For real \(s\geq 0\), we denote the \(L^2\)-based periodic Sobolev spaces by \(H^s((2\pi\TT)^d)^n\): 
	\begin{align*}
		H^s((2\pi\TT)^d)^n &= \{\bU \in L^2((2\pi\TT)^d)^n, \left|\bU \right|_{H^s} <\infty\}, \text{ where }\\
		\left|\bU \right|_{H^s}^2 &= \left|\Lambda^s  \bU \right|_{L^2}^2 = \sum_{\bk\in\ZZ^d} \left\langle \bk \right\rangle^{2s} \left|\widehat{\bU}_{\bk}\right|^2
	\end{align*}
	where \(\widehat{\bU}_{\bk}\) is the \(\bk\)-th fourier coefficient of \(\bU\).
	{For $X$ a Banach space and $I\subset\RR$ an interval, the space of $X$-valued continuous functions on $I$ is denoted $\cC(I;X)$. Given $n\in\NN$, the space of continuously $n$-th differentiable functions is denoted $\cC^n(I;X)$.}
	
	{As mentioned previously, we denote \(\mathcal{T}_N\) the space of trigonometric polynomials of degree \(N\):
	\[\mathcal{T}_N \coloneq \Span\{\exp(\i\bk \cdot \bx), \bk=(k_1,\dots,k_d)\in \ZZ^d, \left|k_j \right|\leq N, j=1,\ldots,d\}, \text{ and } \mathcal{T}_N^{n} \coloneq \underbrace{\mathcal{T}_N \times\ldots\times \mathcal{T}_N}_{n \text{ times}},\]
	and \(\P_N\coloneq \Diag(P_N(D))\) the Fourier multiplier with symbol \(P_N \coloneq \bo_{\llbracket -N,N \rrbracket^d}\) is the \(L^2\)-projection operator onto \(\mathcal{T}_N^{n}\). Here, $\llbracket -N,N \rrbracket\coloneq \{-N,-N+1,\dots,N-1,N\}$.

	We set ${\S_N=\Diag(S_N(D))}$ a Fourier multiplier with symbol $S_N(\cdot)=S(\cdot/N)$ where $S$ is even and satisfies
	\[\begin{cases}
		S(\bk)=1&\text{if $\max_{j=1,\ldots, d}|k_j|\leq 1/2$,}\\
		S(\bk)=0&\text{if $\min_{j=1, \ldots, d}|k_j|\geq 1$,}\\
		S(\bk)\in{[}0,1{]}&\text{otherwise,}
	\end{cases}
	\]
	and $S^{1/2}$ is Lipschitz-continuous. Apart from these properties, the specific profile of the symbol $S$ is inconsequential.

	We denote by \(C(\lambda_1, \lambda_2, \ldots)\) a positive ``constant'' depending on its parameters. Whenever such a parameter represents the norm of a function, \(C\) depends non-decreasingly on said norm. {Whenever the parameter is a subset of the Euclidean space, \(C\) depends non-decreasingly on this parameter when set inclusion is used as (partial) ordering.} Dependency on regularity indices \(s \in \RR\) or the dimension \(d\) are omitted when it is unessential or clear from the context. Note that such constants \(C\) will always be independent of the degree of the approximation \(N\).

	\section{Symmetric quasilinear systems}
	\label{sec.sym}

	In this section, we consider systems~\eqref{eq.hyp}, where for all \(j\in \{1,\ldots, d\}\), \(A_j(\bU)\) satisfies the Assumption~\ref{assump.A1} and is additionally self-adjoint.
	We have the following standard result (see {\em e.g.}~\cite{Benzoni-GavageSerre07}).
	\begin{Proposition}[Well-posedness]\label{prop.sym-hyp}
		Let $s>1+d/2$, and $M>0$. Suppose that for all \(j\in \{1,\ldots, d\}\), \(A_j(\cdot)\) satisfies the Assumption~\ref{assump.A1} and is self-adjoint. There exists $C>0$ and $T>0$ (depending only on $s$ and $M$) such that for every $\bU^0\in H^s((2\pi\TT)^d)^n$ such that $\norm{\bU^0}_{H^s}\leq M$, there exists a unique $\bU\in\cC(I;H^s((2\pi\TT)^d)^n)$ maximal-in-time classical solution to~\eqref{eq.hyp}, and moreover the open time interval $I\supset[0,T/\norm{\bU^0}_{H^s}]$ and for all $0\leq t \leq T/\norm{\bU^0}_{H^s}$,
		\[\norm{\bU}_{H^s}\leq \norm{\bU^0}_{H^s} \exp(C \norm{\bU^0}_{H^s} t)\leq 2\norm{\bU^0}_{H^{s}}.\]
	\end{Proposition}

	\subsection{Discretization with sharp low-pass filters}
	\label{sec.sym.sharp}
	
	Recall that the spatial discretization of the problem~\eqref{eq.hyp} amounts to finding a solution \(\bU_N \colon t\mapsto \mathcal{T}_N^{n}\) to the problem
	\begin{equation}\label{eq.sym-sharp}
		\partial_t \bU_N+\P_N\left(\sum_{j=1}^d A_j(\bU_N)\partial_{x_j}\bU_N\right)=\bz, \quad \bU_N\vert_{t=0} = \P_N\bU^0, 
	\end{equation}
	where \(\P_N = \Diag(P_N(D))\), with \(P_N(\cdot)= \bo_{\llbracket -N,N \rrbracket^d}(\cdot)\).

	We want to show that the semi-discretized solutions \(\bU_N\) to~\eqref{eq.sym-sharp}  converge as \(N \to \infty\) towards \(\bU\) the corresponding solution to the underlying system~\eqref{eq.hyp}. To do so, we will first show that the semi-discretized solutions exist and are bounded on a time interval independent of \(N\). This is Proposition~\ref{prop.sym_sharp_num_bound}. Then we use this bound to compare the semi-discretized and continuous solution on the interval of existence in Proposition~\ref{prop.sym-sharpdifference}. Finally, we refine this result by showing that if \(N\) is large enough the existence of the semi-discretized solution and the estimate on the difference hold on any compact subset of the interval of existence of the solution to~\eqref{eq.hyp}. 

	\begin{Proposition}[Uniform estimates]\label{prop.sym_sharp_num_bound}
		Let $s>d/2+1$ and $M>0$. Suppose that for all \(j\in \{1,\ldots, d\}\), \(A_j(\cdot)\) satisfies the Assumption~\ref{assump.A1} and is self-adjoint. There exists $C>0$ and $T>0$ depending only on \(s\) and \(M\) such that for every $N\in\NN$ and for every $\bU^0\in H^{s}((2\pi\TT)^d)^n$ such that $\norm{\bU^0}_{H^{s}}\leq M$, there exists a unique ${\bU_N\in\cC(I_N;H^s((2\pi\TT)^d)^n)}$ maximal-in-time classical solution to~\eqref{eq.sym-sharp} and $\bU_N\big\vert_{t=0}=\P_N\bU^0$. The open time interval $I_N\supset[0,T/\norm{\bU^0}_{H^{s}}]$ and for all $0\leq t\leq T/\norm{\bU^0}_{H^{s}}$,
		\begin{equation*}
			\norm{\bU_N}_{H^{s}}\leq \norm{\bU^0}_{H^{s}} \exp(C \norm{\bU^0}_{H^{s}} t) \leq 2\norm{\bU^0}_{H^{s}}.
		\end{equation*}

		Moreover, for any \(s'\geq s\), one has $\bU_N\in\cC^1(I_N;H^{s'}((2\pi\TT)^d)^n)$ and for any \(0<T^*\in I_N\) and \(M^*>0\) such that  \(\sup_{t\in [0,T^*]}\left|\bU_N(t,\cdot) \right|_{H^{s}} \leq M^*\) there exists $C^*>0$ depending only on \(s\), \(s'\) and \(M^*\) such that for all \(0\leq t\leq T^*\),
		\[\left| \bU_N \right|_{H^{s'}}\leq \left| \bU^0 \right|_{H^{s'}} \exp(C^* M^* t).\]
	\end{Proposition}
	The key ingredient to show Proposition~\ref{prop.sym_sharp_num_bound} is the following apriori estimate. 
	\begin{Lemma}\label{lem.sym_tder_est}
		Let $s'\geq s>d/2+1$. Suppose that for all \(j\in \{1,\ldots, d\}\), \(A_j(\cdot)\) satisfies the Assumption~\ref{assump.A1} and is self-adjoint. Let \(\bU_N \in \cC(I_N; H^s((2\pi\TT)^d)^n)\) be solution to~\eqref{eq.sym-sharp} on $I_N\subset\RR$ an open time interval. Then $\bU_N \in \cC^1(I_N; H^{s'}((2\pi\TT)^d)^n)$ and for any $t\in I_N$, one has
		\begin{equation}\label{eq.sym_tder_est}
			\frac{\dd}{\dd t}\left|\bU_N \right|_{H^{s'}} \leq C\left( \left|\bU_N \right|_{H^{s}} \right)\left|\bU_N \right|_{H^{s}}\left|  \bU_N\right|_{H^{s'}}.
		\end{equation}
	\end{Lemma}
	\begin{proof}
		Notice $\bU_N=\P_N\bU_N $ and hence we have smoothness in space, $\bU_N\in \cC(I_N; H^\sigma((2\pi\TT)^d)^n)$ for any $\sigma\in\RR$. We infer smoothness in time, $\bU_N\in \cC^1(I_N; H^\sigma((2\pi\TT)^d)^n)$, using eq.~\eqref{eq.sym-sharp}, product and composition estimates in $H^s((2\pi\TT)^d)^n$ ---Propositions~\ref{prop.product_estimates} and~\ref{prop.composition_estimates}--- and that $\P_N:H^s((2\pi\TT)^d)^n\to H^\sigma((2\pi\TT)^d)^n$ is bounded.  Denote \(\dot \bU_N \coloneq \Lambda^{s'} \bU_N\).   Using that \(\bU_N\) satisfies the system~\eqref{eq.sym-sharp}, that \(\P_N\) is symmetric for the $L^2$ inner-product and \({\P_N \dot \bU_N= \dot \bU_N}\), and finally that $A_j$ are self-adoint for all \(j\in \{1,\ldots, d\}\) and integration by parts, we have
		\begin{align*}
			\frac{1}{2} \frac{\dd}{\dd t}\left|\dot\bU_N \right|_{L^2}^2  
			&= \big(\Lambda^{s'} \partial_t \bU_N, \dot \bU_N\big)_{L^2} = - \big( \P_N (\sum_{j= 1}^d [\Lambda^{s'} ,A_j(\bU_N)]\partial_{x_j} \bU_N), \dot \bU_N\big)_{L^2} - \big( \P_N (\sum_{j= 1}^d A_j(\bU_N)\partial_{x_j} \dot \bU_N), \dot \bU_N\big)_{L^2}\\
			& =- \big(\sum_{j= 1}^d [\Lambda^{s'} ,A_j(\bU_N)]\partial_{x_j} \bU_N, \dot \bU_N\big)_{L^2} - \big(\sum_{j= 1}^d A_j(\bU_N)\partial_{x_j} \dot \bU_N, \dot \bU_N\big)_{L^2}\\
			&=- \big(\sum_{j= 1}^d [\Lambda^{s'} ,A_j(\bU_N)]\partial_{x_j} \bU_N, \dot \bU_N\big)_{L^2} + \frac{1}{2}\big(\sum_{j= 1}^d [\partial_{x_j},A_j(\bU_N)] \dot \bU_N,  \dot \bU_N\big)_{L^2}.
		\end{align*}
		Using the commutator and composition estimates of Propositions~\ref{prop.commutator_s},\ref{prop.composition_estimates} with $s_0=s-1$ for the first contribution, the continuous Sobolev embedding $H^{s}((2\pi\TT)^d)^n\subset W^{1,\infty}((2\pi\TT)^d)^n$ (Proposition~\ref{prop.embedding}) for the second as well the Cauchy--Schwarz inequality gives the bound
		\begin{equation*}
			\frac{1}{2} \frac{\dd}{\dd t}\left|\dot\bU_N \right|_{L^2}^2 \leq C( \left|\bU_N \right|_{H^{s}} )\left|\bU_N \right|_{H^{s}}\left| \dot \bU_N\right|_{L^2}^2,
		\end{equation*}
		which yields the desired inequality.
	\end{proof}
	Now we proceed to prove Proposition~\ref{prop.sym_sharp_num_bound}. 

	\begin{proof}[Proof of Proposition~\ref{prop.sym_sharp_num_bound}.]
		For each \(N>0\), existence and uniqueness of a local-in-time solution \(\bU_N\) to~\eqref{eq.sym-sharp} follows from its formulation as a system of ODEs in the Banach space $H^s((2\pi\TT)^d)^n$ (using eq.~\eqref{eq.sym-sharp}, product and composition estimates in $H^{s-1}((2\pi\TT)^d)^n$ ---Propositions~\ref{prop.product_estimates} and~\ref{prop.composition_estimates}--- and that $\P_N:H^{s-1}((2\pi\TT)^d)^n\to H^s((2\pi\TT)^d)^n$ is bounded) and the Picard--Lindelöf theorem. We denote \(I_N\) the maximal interval of existence, and notice as in the proof of Lemma~\ref{lem.sym_tder_est} that $\bU_N\in \cC^1(I_N; H^{s'}((2\pi\TT)^d)^n)$ for any $s'\geq s$.
		
		The second part of the proposition is an immediate consequence of Lemma~\ref{lem.sym_tder_est} and Grönwall's inequality, using that $\left| \bU_N\vert_{t=0} \right|_{H^{s'}}=\left|\P_N \bU^0 \right|_{H^{s'}}\leq \left|\bU^0 \right|_{H^{s'}}$. To show the first part of the proposition, we use a standard continuity argument. Let
		\[\varphi_N:t\mapsto \sup_{t'\in[0,t]}\left| \bU_N(t',\cdot) \right|_{H^{s}} \quad \text{ and } \quad J_N\coloneq\big\{t\in I_N\cap\RR^+ \ : \  \varphi_N(t)\leq 2\left|\bU^0 \right|_{H^s}\big\}.\]
		Since $\bU_N\in \cC(I_N; H^{s}((2\pi\TT)^d)^n)$, $\varphi_N\in \cC(I_N\cap \RR^+; \RR)$ is non-decreasing and  $J_N=\varphi_N^{-1}([0,2\left|\bU^0 \right|_{H^s}])$ is a closed interval. Let us prove that one can set $T>0$ independently of $N$ such that $J_N\cap [0, T/\left|\bU^0 \right|_{H^{s}}]$ is a non-empty open subset of $[0, T/\left|\bU^0 \right|_{H^{s}}]$. Notice $0\in J_N$. Let $t\in J_N\cap [0, T/\left|\bU^0 \right|_{H^{s}}]$ . By the second part of the proposition, {we can set $C=2C^*$} depending only on $s$ and $\left|\bU^0 \right|_{H^s}$ such that
		\begin{equation}\label{eq.sym_sharp_est1}
			\left| \bU_N(t,\cdot) \right|_{H^s} \leq \left| \bU^0 \right|_{H^{s}} \exp\left( C\left|\bU^0 \right|_{H^s}t \right) \leq \left| \bU^0 \right|_{H^{s}}\exp(CT).
		\end{equation}
		Choosing $T=\ln(3/2)/C$, we find that $\varphi_{N}(t)\leq \frac32 \left| \bU^0 \right|_{H^{s}}$, and hence ---by the continuity and monotonicity of $\varphi_N$--- there exists $\delta>0$ such that $[0,t+\delta]\in J_N$. This concludes the proof that $J_N\cap [0, T/\left|\bU^0 \right|_{H^{s}}]$ is a non-empty, connected, closed and open subset of $[0, T/\left|\bU^0 \right|_{H^{s}}]$. Hence $I_N\supset J_N\supset [0, T/\left|\bU^0 \right|_{H^{s}}]$ and estimate~\eqref{eq.sym_sharp_est1} concludes the proof.
	\end{proof}

	Having established a bound on the semi-discretized solution \(\bU_N\), we proceed to estimate the difference between the solution \(\bU_N\) to~\eqref{eq.sym-sharp} and \(\bU\), solution to~\eqref{eq.hyp}. 

	\begin{Proposition}[Convergence]\label{prop.sym-sharpdifference}
		Let $s>d/2+1$, \(M>0\). Suppose that \(\bU^0\in H^s((2\pi\TT)^d)^n\) with \({\norm{\bU^0}_{H^s}\leq M}\). Suppose that for all \(j\in \{1,\ldots, d\}\), \(A_j(\cdot)\) satisfies the Assumption~\ref{assump.A1} and is self-adjoint. Denote $\bU\in\cC(I;H^s((2\pi\TT)^d)^n)$ the maximal-in-time classical solution to~\eqref{eq.hyp}, and $\bU_N\in\cC(I_N;H^s((2\pi\TT)^d)^n)$ the maximal-in-time classical solution to~\eqref{eq.sym-sharp}. Let $T>0$ be the minimum value of Propositions~\ref{prop.sym-hyp} and~\ref{prop.sym_sharp_num_bound}. For all \({0\leq t\leq T/\norm{\bU^0}_{H^s}}\), there is a $C>0$, depending only on $s$ and \(M\) such that for any $0\leq r\leq s$,
		\[\norm{(\bU-\bU_N)(t,\cdot)}_{H^r}\leq C\, \norm{\bU^0}_{H^s}\, N^{r-s} .\]
		Moreover, for every compact subset \(I^* \subset I\), there is an \(N_0\in\NN\) and $C^*>0$, depending only on $s$, $|I^*|$ and $M^*\coloneq\sup_{t\in I^*} \left|\bU(t,\cdot)\right|_{H^s}$ such that for all $N\geq N_0$, $I_N\supset I^*$ and for any $0\leq r\leq s$,
		\[\sup_{t\in I^*}\norm{(\bU-\bU_N)(t,\cdot)}_{H^r}\leq C^*\, M^*\, N^{r-s}.\]
	\end{Proposition}
	\begin{proof} 		
		We shall first prove the result for $t \in [0,T/\norm{\bU^0}_{H^s}]$. Let us assume first that $\bU^0\in H^{s+1}((2\pi\TT)^d)^n$ so that $\bU\in\cC(I;H^{s+1}((2\pi\TT)^d)^n)\cap \cC^1(I;H^{s}((2\pi\TT)^d)^n) $, the general case being deduced afterwards. Recall that $T>0$ is the minimum value of Propositions~\ref{prop.sym-hyp} and~\ref{prop.sym_sharp_num_bound}, so that we have \(\left|\bU(t,\cdot) \right|_{H^s}\leq 2 \left|\bU^0 \right|_{H^s}\) and \(\left|\bU_N(t,\cdot) \right|_{H^s}\leq 2 \left|\bU^0 \right|_{H^s}\) for all $ N\in\NN$. 
		Denote $\bD_N\coloneq\bU-\bU_N$ and notice
		\[\partial_t \bD_N + \sum_{j=1}^d A_j(\bU_N)\partial_{x_j}\bD_N+\sum_{j=1}^d (A_j(\bU)-A_j(\bU_N))\partial_{x_j}\bU=-(\Id-\P_N)\left(\sum_{j=1}^d A_j(\bU_N)\partial_{x_j}\bU_N\right).\]
		Now we apply the smooth low-pass filter \(\S_N\) and use that $\S_N(\Id-\P_N)=0$:
		\begin{multline*}
			\partial_t \S_N\bD_N + \S_N^{1/2}\left(\sum_{j=1}^d A_j(\bU_N)\partial_{x_j}\S_N^{1/2}\bD_N\right)+ \S_N^{1/2}\left(\sum_{j=1}^d[\S_N^{1/2}, A_j(\bU_N)]\partial_{x_j}\bD_N\right)\\+\S_N\left(\sum_{j=1}^d (A_j(\bU)-A_j(\bU_N))\partial_{x_j}\bU\right)=\bz.	
		\end{multline*} 
		Testing against $\bD_N$, using that $\S_N$ is symmetric for the $L^2((2\pi\TT)^d)^n$ inner-product and commutes with \(\partial_t\) and after some rearranging we get
		\begin{align*}
			\frac{\dd}{\dd t}\left| \S_N^{1/2}\bD_N\right|_{L^2}^2 = - \left(\sum_{j=1}^d A_j(\bU_N)\partial_{x_j}\S_N^{1/2}\bD_N,  \S_N^{1/2}\bD_N \right)_{L^2}- \left( \sum_{j=1}^d[\S_N^{1/2}, A_j(\bU_N)]\partial_{x_j}\bD_N,  \S_N^{1/2}\bD_N \right)_{L^2}\\
			- \left( \S_N^{1/2}\left(\sum_{j=1}^d (A_j(\bU)-A_j(\bU_N))\partial_{x_j}\bU\right),  \S_N^{1/2}\bD_N \right)_{L^2}.
		\end{align*}
		Using integration by parts and the properties of \(A_j(\cdot)\) ---Assumption~\ref{assump.A1} together with the composition estimate of Proposition~\ref{prop.composition_estimates} and self-adjointness--- and the continuous Sobolev embedding $H^s((2\pi\TT)^d)^n\subset W^{1,\infty}((2\pi\TT)^d)^n$ (Proposition~\ref{prop.embedding}), the first term on the right-hand side is bounded by  \(C(\left|\bU_N \right|_{H^{s}})\left|\bU_N \right|_{H^{s}}\left|\S_N^{1/2}\bD_N \right|_{L^2}^2\). Using the commutator estimate of Proposition~\ref{prop.commutator_0} and the composition estimate of Proposition~\ref{prop.composition_estimates}, the second term on the right is bounded by \(C(\left|\bU_N \right|_{H^{s}})\left|\bU_N \right|_{H^{s}}\left|\bD_N \right|_{L^2}\left|\S_N^{1/2}\bD_N \right|_{L^2}\). By the composition estimate of Proposition~\ref{prop.composition_estimates}, the third term is bounded by \(C(\left|\bU \right|_{H^{s}}, \left|\bU_N \right|_{H^s})\left|\bU \right|_{H^{s}}\left|\bD_N \right|_{L^2}\left|\S_N^{1/2}\bD_N \right|_{L^2}\), where we again use Assumption~\ref{assump.A1} as well as the boundedness of \(\S_N^{1/2}\). In all three estimates, we also used the Cauchy--Schwarz inequality. Altogether, we get
		\[\frac12\frac{\dd}{\dd t}\big( \norm{\S_{N}^{1/2}\bD_N}_{L^2}^2 \big) \leq C(\norm{\bU}_{H^s},\norm{\bU_N}_{H^s})\left( \norm{\bU_N}_{H^s} + \norm{\bU}_{H^s} \right)\big(\norm{\S_{N}^{1/2}\bD_N}_{L^2}+\norm{\bD_N}_{L^2}\big)\norm{\S_{N}^{1/2}\bD_N}_{L^2}.\]
		Now we remark that since the symbol $S_N$ satisfies $S_N(\cdot)\in[0,1]$ and $S_N(\bk)=1$ if $\max_{j=1,\dots,d}|k_j|\leq N/2$ we have \(\big\vert (1-S_N^{1/2}(\cdot))\langle \cdot\rangle^{-s}\big\vert_{L^\infty} \leq \langle N/2\rangle^{-s}\) and hence
		\[\norm{\bD_N}_{L^2}\leq \norm{\S_{N}^{1/2}\bD_N}_{L^2}+\norm{(\Id-\S_{N}^{1/2})\bD_N}_{L^2}\leq \norm{\S_{N}^{1/2}\bD_N}_{L^2}+\langle N/2\rangle^{-s}\norm{\bD_N}_{H^s}.\]
		Hence since $\norm{\bD_N}_{H^s}\leq \norm{\bU}_{H^s}+\norm{\bU_N}_{H^s}\leq 4 \norm{\bU^0}_{H^s}$ by the triangle inequality we have
		\[\frac12\frac{\dd}{\dd t}\big( \norm{\S_{N}^{1/2}\bD_N}_{L^2}^2 \big) \leq C(\norm{\bU^0}_{H^s})\left|\bU^0 \right|_{H^s}\big(\norm{\S_{N}^{1/2}\bD_N}_{L^2}+ \norm{\bU^0}_{H^s}N^{-s}\big)\norm{\S_{N}^{1/2}\bD_N}_{L^2}\]
		and we infer by Grönwall's Lemma that
		\[\norm{\S_{N}^{1/2}\bD_N(t,\cdot)}_{L^2}\leq C(\norm{\bU^0}_{H^s})\left|\bU^0 \right|_{H^s}\big(\norm{\S_{N}^{1/2}\bD_N\big\vert_{t=0}}_{L^2} +t  N^{-s}\norm{\bU^0}_{H^s}\big) \, \exp\Big(C(\norm{\bU^0}_{H^s})\left|\bU^0 \right|_{H^s} t\Big).\]
		The desired estimate for $r=0$ follows by using that $\S_{N}^{1/2}\bD_N\big\vert_{t=0}=\S_{N}^{1/2}(\Id-\P_N)\bU^0=0$, $t \in [0,T/\norm{\bU^0}_{H^s}]$ and again 
		\[\norm{\bD_N}_{L^2}\leq \norm{\S_{N}^{1/2}\bD_N}_{L^2}+\langle N/2\rangle^{-s}\norm{\bU^0}_{H^s}.\]
		The general case $0\leq r\leq s$ follows by the interpolation inequality, Proposition~\ref{prop.interpolation_inequality}, and using once again that $\norm{\bU-\bU_N}_{H^s}\leq 4\norm{\bU^0}_{H^s}$ by the triangle inequality.
		
		Let us now explain why the same result holds in the general case $\bU^0\in H^{s}((2\pi\TT)^d)^n$. Consider $(\bU^0_k)_{k\in\NN}$ a sequence (constructed by Fourier truncation) such that for all $k\in\NN$, $\bU^0_k\in H^{s+1}((2\pi\TT)^d)^n$ and $\bU^0_k \to \bU^0$ in $ H^{s}((2\pi\TT)^d)^n$ as $k\to\infty$, and $\norm{\bU^0_k}_{H^s}\leq \norm{\bU^0}_{H^s}$. Then we can apply the above for each $k\in\NN$ and infer that $\bU_k$ (respectively $\bU_{k,N}$) the solutions to~\eqref{eq.hyp} (respectively~\eqref{eq.sym-sharp}) emerging from the initial data $\bU^0_k $ satisfy for all \({0\leq t\leq T/\norm{\bU^0}_{H^s}}\) and for any $0\leq r\leq s$,
		\[\norm{(\bU_k-\bU_{k,N})(t,\cdot)}_{H^r}\leq C\, \norm{\bU^0}_{H^s}\, N^{r-s} ,\]
		where $C>0$ depends only on $s$ and \(M\), and in particular is uniform with respect to $k$. We now pass to the limit as $k\to\infty$. By standard estimates on the linearized systems from~\eqref{eq.hyp} and~\eqref{eq.sym-sharp} (see {\em e.g.}~\cite[Proposition~7.1.8]{Metivier08}), we have $\bU_k\to \bU$ and $\bU_{k,N}\to \bU_N$ in $\cC([0,T/\norm{\bU^0}_{H^s}];L^2((2\pi\TT)^d)^n)$ as $k\to\infty$, where we denote  $\bU\in \cC([0,T/\norm{\bU^0}_{H^s}];H^s((2\pi\TT)^d)^n)$ (respectively $\bU_{N}\in\cC([0,T/\norm{\bU^0}_{H^s}];H^s((2\pi\TT)^d)^n)$) the solutions to~\eqref{eq.hyp} (respectively~\eqref{eq.sym-sharp}) emerging from the initial data $\bU^0 $, as in the statement of the Proposition. Because the above estimate is uniform with respect to $k$, we infer as desired that the limits satisfy for all \({0\leq t\leq T/\norm{\bU^0}_{H^s}}\) and for any $0\leq r\leq s$,
		\[\norm{(\bU-\bU_{N})(t,\cdot)}_{H^r}\leq C\, \norm{\bU^0}_{H^s}\, N^{r-s} .\]

		Let us now prove the proposition for general $I^*$ compact subset of $I$. Without loss of generality, we can assume $0\in I^*$ and we will focus on positive times, $t\in I^*\cap \RR^+$, negative times being obtained by time-symmetry. Let $d/2+1<s_0<s$. Denote $M^*=2\sup_{t\in I^*}\left|\bU\right|_{H^{s_0}}$ and $C^*$ the constant depending on $s_0,s,M^*$ as in Proposition~\ref{prop.sym_sharp_num_bound} with $s_0$ playing the role of $s$ and $s$ playing the role of $s'$, and $M=2\left|\bU^0\right|_{H^{s}}\exp(C^*M^* T^*)$, where \(T^* \coloneq \sup(I^*)\). We set
		\[J_N=\big\{ t\in I_N\cap\RR^+ \ : \ \sup_{t'\in[0,t]}\left|\bU_N(t',\cdot)\right|_{H^{s_0}}\leq M^* \text{ and } \sup_{t'\in[0,t]}\left|\bU_N(t',\cdot)\right|_{H^{s}}\leq M\big\},\]
		and our aim is to prove that there exists $N_0\in\NN$ such that for any $N\geq N_0$, $J_N\supset I^*\cap\RR^+$. We use the continuity argument. By continuity of $\bU_N\in \cC(I_N;H^{s}((2\pi\TT)^d)^n)$, we have that $J_N$ is a non-empty closed interval. Let us now prove that $J_N\cap I^*\cap \RR^+$ is an open subset of  $I^*\cap \RR^+$. Let $t\in J_N\cap I^*\cap \RR^+$. We can follow the proof of the first part of the proposition replacing the bound $\norm{\bU}_{H^s}+\norm{\bU_N}_{H^s}\leq 4 \norm{\bU^0}_{H^s}$ with $\norm{\bU}_{H^{s}}+\norm{\bU_N}_{H^{s}}\leq M^*+M$ to infer that there exists $C$, depending only on $s,M^*$ and $T^*$ such that
		for all $0\leq r\leq s$,
		\[\sup_{t'\in[0,t]}\norm{\bU(t',\cdot)-\bU_N(t',\cdot)}_{H^r}\leq C (M^*+M)N^{r-s} .\]
		Applying this estimate with $r=s_0$, it follows that there exists $N_0\in\NN$ depending only on $s_0,s,M^*$ and $T^*$ such that for any $N\geq N_0$, 
		\[ \sup_{t'\in[0,t]}\norm{\bU_N(t',\cdot)}_{H^{s_0}}\leq \sup_{t'\in[0,t]}\norm{\bU(t',\cdot)}_{H^{s_0}}+ \sup_{t'\in[0,t]}\norm{\bU(t',\cdot)-\bU_N(t',\cdot)}_{H^{s_0}}\leq \frac23 M^* .\]
		Moreover, using the second part of Proposition~\ref{prop.sym_sharp_num_bound} (recall $s_0$ playing the role of $s$ and $s$ playing the role of $s'$), we have
		\[ \sup_{t'\in[0,t]}\norm{\bU_N(t',\cdot)}_{H^{s}}\leq  \left| \bU^0 \right|_{H^{s}} \exp(C^* M^* t) \leq  \frac12 M .\]
		\sloppy{This shows, using again the continuity of $\bU_N\in \cC(I_N;H^{s}((2\pi\TT)^d)^n)$, that there exists $\delta>0$ such that ${[0,t+\delta]\subset J_N}$, and hence that $J_N\cap I^*\cap \RR^+$ is a non-empty, connected, closed and open subset of $I^*\cap \RR^+$. Hence $ J_N\supset  I^*\cap \RR^+$. Moreover, the desired estimate has been proven along the argument. This concludes the proof.}
	\end{proof}

	\subsection{Discretization with smooth low-pass filters}
	\label{sec.sym.smooth}
	
	Consider now solutions to the systems semi-discretized with a smooth low-pass filter, given by~\eqref{eq.hyp-smooth0} and~\eqref{eq.hyp-smooth}, and which we recall here for the sake of clarity. 
	\begin{equation}\label{eq.sym-smooth0}
		\partial_t \bU_N+\S_N\left(\sum_{j=1}^d A_j(\bU_N)\partial_{x_j}\bU_N\right)=\bz, \quad \bU_N\vert_{t=0} = \P_N\bU^0,
	\end{equation}
	
	\begin{equation}\label{eq.sym-smooth}
		\partial_t \bU_N+\sum_{j=1}^d (A_j^0 + \S_N (A_j^1(\bU_N)[\circ]))\partial_{x_j}\bU_N=\bz, \quad \bU_N\vert_{t=0} = \P_N \bU^0,
	\end{equation}
	where \(\S_N\) is the smooth low-pass filter described in Section~\ref{sec.intro}, and $A_j(\cdot)=A_j^0 +A_j^1(\cdot)$, $A_j^0=A_j(\bz)$.

	For symmetric systems, there is no great difference between the semi-discretization with sharp versus smooth low-pass filters. The results are the same, although the proofs must be adapted slightly.  We outline the results and proofs, but refer to the previous section for technical details. 
	
	Our result regarding existence and boundedness of solutions \(\bU_N\) to the semi-discrete problems (uniformly with respect to \(N\)) extends to the case of smooth low-pass filters.
	\begin{Proposition}[Uniform estimates]\label{prop.sym-smoothnum-bound}
		The statement of Proposition~\ref{prop.sym_sharp_num_bound} holds replacing~\eqref{eq.sym-sharp} with~\eqref{eq.sym-smooth0} or~\eqref{eq.sym-smooth}.
	\end{Proposition}
	\begin{proof}
		We consider first the system~\eqref{eq.sym-smooth}. As for the case with the sharp low-pass filter, the key ingredient is an apriori estimate 
				\begin{equation}\label{eq.sym-smoothapriori}
					\frac{\dd}{\dd t}\left|\bU_N \right|_{H^{s'}} \leq C(\left|\bU_N \right|_{H^{s}})\left|\bU_N \right|_{H^{s}}\left|  \bU_N\right|_{H^{s'}}
				\end{equation}
		for $s'\geq s>d/2+1$ and \(\bU_N \in \cC(I_N; H^s((2\pi\TT)^d)^n)\) solution to~\eqref{eq.sym-smooth} on $I_N\subset\RR$ open time interval.
		To show this we first notice that, applying $(\Id-\P_N)$ to~\eqref{eq.sym-smooth} and using that $(\Id-\P_N)\S_N=0$, we have $\partial_t (\Id-\P_N)\bU_N+\sum_{j=1}^d A_j^0 \partial_{x_j}(\Id-\P_N)\bU_N=0$ and $(\Id-\P_N)\bU_N\vert_{t=0}=0$. By uniqueness of the solution to this initial-value problem, we infer $\bU_N=\P_N\bU_N$ and hence $\bU_N \in \cC^1(I; H^\sigma((2\pi\TT)^d)^n)$ for all $\sigma\in\RR$. Then we apply \(\Lambda^{s'}\) to the system~\eqref{eq.sym-smooth}, denote $\dot \bU_N \coloneq\Lambda^{s'}\bU_N$ and infer  
		\begin{multline}\label{eq.est3}
			\partial_t \dot\bU_N+\sum_{j=1}^d (A_j^0 + \S_N^{1/2} A_j^1(\bU_N)[\S_N^{1/2} \circ])\partial_{x_j}\dot\bU_N\\
			=-\S_N\left(\sum_{j=1}^d [\Lambda^{s'},A_j^1(\bU_N)]\partial_{x_j}\bU_N\right)-\S_N^{1/2}\left(\sum_{j=1}^d [\S_N^{1/2},A_j^1(\bU_N)]\partial_{x_j}\dot\bU_N\right).	
		\end{multline}
		\sloppy{Using that \(\S_N:L^2\to L^2\) is bounded, as well as the product, composition and commutator estimates ---Propositions~\ref{prop.product_estimates},~\ref{prop.composition_estimates},~\ref{prop.commutator_s} and~\ref{prop.commutator_0} --- the terms on the right-hand side can be estimated in \(L^2((2\pi\TT)^d)^n\) as}
		\begin{equation*}
			\left|{\rm RHS} \right|_{L^2} \leq C(\left|\bU_N \right|_{H^{s}})\left|\bU_N \right|_{H^{s}}\left|\dot\bU_N \right|_{L^2}.
		\end{equation*}		
		By assumption, \(A_j(\bU)=A_j^0 + A_j(\bU)\) is a symmetric matrix for all $\bU\in\RR^n$ and \(j\in \{1,\ldots,d\}\). In particular, this implies that \(A_j^0=A_j(\bz)\) and hence \(A_j^1(\bU)\) are both symmetric matrices, which in turn implies that the operator \(A_j^0 + S_N^{1/2}A_j^1(\bU_N)\S_N^{1/2}\) is symmetric for the $L^2((2\pi\TT)^d)^n$ inner-product for all \(j\in \{1,\ldots,d\}\). Arguing in the usual manner using integration by parts, we have 
		\begin{align*}
			\left( \sum_{j=1}^d (A_j^0 + \S_N^{1/2} A_j^1(\bU_N)[\S_N^{1/2} \circ])\partial_{x_j}\dot\bU_N, \dot \bU_N \right)_{L^2} = - \frac{1}{2}\sum_{j=1}^d\left( [\partial_{x_j}, \S_N^{1/2} A_j^1(\bU_N) \S_N^{1/2}]\dot \bU_N, \dot \bU_N\right)_{L^2} \\
			\leq C(\left|\bU_N \right|_{H^{s}})\left|\bU_N \right|_{H^{s}}\left|\dot \bU_N \right|_{L^2}^2,
		\end{align*}
		where we used the continuous Sobolev embedding $H^s((2\pi\TT)^d)^n\subset W^{1,\infty}((2\pi\TT)^d)^n$ (Proposition~\ref{prop.embedding}) and that $\S_N^{1/2}:L^2\to L^2$ is bounded.
		Testing the identity~\eqref{eq.est3} against \(\dot \bU_N\), using the Cauchy--Schwarz inequality on the right-hand side and inserting the two estimates above yields the desired inequality~\eqref{eq.sym-smoothapriori}. 
		With this estimate in hand, the rest of the proof follows exactly the proof of Proposition~\ref{prop.sym_sharp_num_bound}.
		
		Considering now the system~\eqref{eq.sym-smooth0}, we notice that the inequality~\eqref{eq.sym-smoothapriori} is obtained as above, using additionally that $\S_NA_j^0=\S_N^{1/2}A_j^0 \S_N^{1/2}$ is self-adjoint, and the result follows.
 	\end{proof}
	
	The spectral convergence of solutions \(\bU_N\) to the semi-discrete problems towards the corresponding solution to the continuous problem~\eqref{eq.hyp} as \(N\to\infty\) is the identical in the case of sharp or smooth low-pass filters.
	\begin{Proposition}[Convergence]\label{prop.sym-smoothdifference}
				The statement of Proposition~\ref{prop.sym-sharpdifference} holds replacing~\eqref{eq.sym-sharp} with~\eqref{eq.sym-smooth0} or~\eqref{eq.sym-smooth}.
	\end{Proposition}
	\begin{proof}
		The proof follows the proof of Proposition~\ref{prop.sym-sharpdifference}, with one modification. Consider \(\bU_N \in \cC(I_N; H^s((2\pi\TT)^d)^n)\) solution to~\eqref{eq.sym-smooth} (considering instead the solution to~\eqref{eq.sym-smooth0} amounts to replacing $A_j^1(\bU_N)$ with $A_j(\bU_N)$ in the right-hand side of the following identity, with no consequence).
		Then the difference $\bD_N\coloneq\bU-\bU_N$ satisfies
		\[\partial_t \bD_N + \sum_{j=1}^d A_j(\bU_N)\partial_{x_j}\bD_N+\sum_{j=1}^d (A_j(\bU)-A_j(\bU_N))\partial_{x_j}\bU=(\Id-\S_N)\left(\sum_{j=1}^d A_j^1(\bU_N)\partial_{x_j}\bU_N\right).\]
		Instead of applying \(\S_N\) as in the proof of Proposition~\ref{prop.sym-sharpdifference}, we apply \(\S_{N/2}\), noting that \(\S_{N/2}(\Id - \S_N) = 0\). Then we have
		\begin{multline*}\partial_t \S_{N/2}\bD_N + \S_{N/2}^{1/2}\left(\sum_{j=1}^d A_j(\bU_N)\partial_{x_j}\S_{N/2}^{1/2}\bD_N\right)+ \S_{N/2}^{1/2}\left(\sum_{j=1}^d[\S_{N/2}^{1/2}, A_j(\bU_N)]\partial_{x_j}\bD_N\right)\\
			+\S_{N/2}\left(\sum_{j=1}^d (A_j(\bU)-A_j(\bU_N))\partial_{x_j}\bU\right)=\bz.
		\end{multline*}
		Because \(\S_{N/2}\) satisfies the same commutator estimates as \(\S_N\) and $\|\Id-\S_{N/2}^{1/2}\|_{H^s\to L^2}\leq \langle N/4\rangle^{-s}$, we may then proceed exactly as in the proof of Proposition~\ref{prop.sym-sharpdifference}.
	\end{proof}

	\section{Symmetrizable quasilinear systems}
	\label{sec.symable}

	In this section, we consider systems that are symmetrizable in the sense of Friedrichs. That is, we assume that there is an open set \(\cU\subset \RR^{n}\) containing the origin and an operator \(S(\bU)\) which is a symmetrizer for the system~\eqref{eq.hyp}:
	
	\begin{Assumption}{A.2}\label{assump.S1}
	There exists \(S(\cdot)\) a Friedrichs-symmetrizer for the system~\eqref{eq.hyp}, that is, there exists an open set \({\cU \subset \RR^{n}}\) with \(\bz\in\cU\) 
		such that for all $\bU \in\cU$,  $S(\bU)$ is real-valued, symmetric positive definite, and for all \(j\in \{1, \ldots, d\}\), \(S(\bU)A_j(\bU)\) is symmetric.
				{We assume that all entries of $S(\cdot)$ are polynomial.}
	\end{Assumption}
	\begin{Remark}
		As we assumed in Assumption~\ref{assump.A1} that entries of $A_j(\bU)$ are polynomial for all \(j\in \{1, \ldots, d\}\), then if Friedrichs-symmetrizers exist it is always possible to select one whose entries are polynomial. Indeed, the assumptions that $S(\cdot)$ is real-valued and symmetric and that for all \(j\in \{1, \ldots, d\}\), \(S(\cdot)A_j(\cdot)\) is symmetric constitutes a system of linear equations for entries of $S$. By considering the corresponding matrix in the field of real rational fractions (since entries of $A_j(\bU)$ are real polynomials) and performing Gaussian elimination, we see that the system of linear equations can be solved for (non-identically zero) smooth functions if and only if it can be solved for rational fractions, which then can be chosen polynomials after multiplication of all entries by a common multiple of all denominators. Notice however that the domain of hyperbolicity defined as the open set on which $S(\cdot)$ is positive definite depends on the choice of the symmetrizer.
	\end{Remark}

	We have the following standard result~\cite{Benzoni-GavageSerre07}.
	\begin{Proposition}[Well-posedness]\label{prop.symable-hyp}
		Suppose that the system~\eqref{eq.hyp} satisfies Assumptions~\ref{assump.A1} and~\ref{assump.S1}. Let $s>1+d/2$, $M>0$ and $\cK\subset \cU$ compact.  There exists $C>0$ and $T>0$ (depending only on $s, M$ and \(\cK\)) such that for every $\bU^0\in H^s((2\pi\TT)^d)^n$ such that $\norm{\bU^0}_{H^s}\leq M$ and taking values in \(\cK\), there exists a unique $\bU\in\cC(I;H^s((2\pi\TT)^d)^n)$ maximal-in-time classical solution to~\eqref{eq.hyp} and $\bU\big\vert_{t=0}=\bU^0$, and moreover the open time interval $I\supset[0,T/\norm{\bU^0}_{H^s}]$ and for all $t\in[0,T/\norm{\bU^0}_{H^s}]$,
		\[\norm{\bU}_{H^s}\leq C \norm{\bU^0}_{H^s} \exp(C \norm{\bU^0}_{H^s} t).\]
	\end{Proposition}

	\subsection{Discretization with smooth low-pass filters}
	\label{sec.symable.smooth}
		
	We consider the spatial discretization with smooth low-pass filters first, as it is more similar to the previous section. 
	Recall that the spatial discretization with the smooth low-pass filter \(\S_N\) amounts to solving
	\begin{equation}\label{eq.symable-smooth0}
		\partial_t \bU_N+\S_N\left(\sum_{j=1}^d A_j(\bU_N)\partial_{x_j}\bU_N\right)=\bz, \quad \bU_N\vert_{t=0} = \P_N\bU^0.
	\end{equation}
	As discussed in the introduction, one would typically prefer in practice the variant 
	\begin{equation}\label{eq.symable-smooth}
		\partial_t \bU_N+\sum_{j=1}^d (A_j^0 + \S_N (A_j^1(\bU_N)[\circ]))\partial_{x_j}\bU_N=\bz, \quad \bU_N\vert_{t=0} = \P_N \bU^0,
	\end{equation}
	where \(A_j^0\coloneq A_j(\bz)\) and \(A_j^1(\cdot) \coloneq A_j(\cdot)- A_j^0 \). 
	However, we face a difficulty that while any symmetrizer for the system~\eqref{eq.hyp},	\(S(\cdot)\), readily provides a suitable symmetrizer of the  semi-discretized system~\eqref{eq.symable-smooth0}, such is not the case for system~\eqref{eq.symable-smooth}, and additional assumptions are needed. 
	
	\begin{Assumption}{A.3}[Compatibility of the symmetrizer]\label{assump.AS}
		Supposing the Assumptions~\ref{assump.A1} and~\ref{assump.S1} hold, and decomposing 
		\(A_j(\cdot) =  A_j^0+A_j^1(\cdot)\) where $A_j^0\coloneq A_j(\bz)$ and $S(\cdot) =  S^0+S^1(\cdot)$ where $S^0\coloneq S(\bz)$, 
		we have 
		\[\forall j\in\{1, \ldots, d\}, \ \forall \bU\in\cU, \quad S^0A_j^0, \ S^0A_j^1(\bU)+S^1(\bU)A_j^0 \text{ and } S^1(\bU)A_j^1(\bU) \text{ are symmetric.}\]
	\end{Assumption}
	\begin{Remark}\label{rem.AS-linear}
		As can be seen by Taylor-expanding $S(\bU)A_j(\bU)$ about the origin, Assumption~\ref{assump.AS} holds in particular when $A_j(\cdot)$ and $S(\cdot)$ are linear, that is entries of $A_j^1(\cdot)$ and $S^1(\cdot)$ are homogeneous polynomials of degree 1; see also Remark~\ref{rem.general1}.
	\end{Remark}
	
	We have the following bound on solutions \(\bU_N\) to the semi-discretized problems~\eqref{eq.symable-smooth0} or~\eqref{eq.symable-smooth}.
	\begin{Proposition}[Uniform estimates]\label{prop.symable_smooth_num_bound}
		Suppose that the system~\eqref{eq.hyp} satisfies Assumptions~\ref{assump.A1} and~\ref{assump.S1}. Let ${s>d/2+1}$, $M>0$ and $\cK\subset \cU$ compact. There exists $C>0$ and $T>0$ (depending only on \(s,M,\) and \(\cK\)) such that for every $N\in\NN$ and for every $\bU^0\in H^{s}((2\pi\TT)^d)^n$ such that $\norm{\bU^0}_{H^{s}}\leq M$ and taking values in $\cK$, there exists a unique $\bU_N\in\cC(I_N;H^s((2\pi\TT)^d)^n)$ maximal-in-time classical solution to~\eqref{eq.symable-smooth0} and $\bU\big\vert_{t=0}=\P_N\bU^0$. The open time interval $I_N\supset[0,T/\norm{\bU^0}_{H^{s}}]$ and for all $0\leq t\leq T/\norm{\bU^0}_{H^{s}}$,
		\begin{equation*}
			\norm{\bU_N}_{H^{s}}\leq C\norm{\bU^0}_{H^{s}} \exp(C \norm{\bU^0}_{H^{s}} t) \leq 2 C \norm{\bU^0}_{H^{s}}.
		\end{equation*}
		
		Moreover, for any \(s'\geq s\), one has $\bU_N\in\cC^1(I_N;H^{s'}((2\pi\TT)^d)^n)$ and for any \(0<T^*\in I_N\),  \(M^*>0\)  and $\cK^*\subset\cU$ compact such that  \(\sup_{t\in [0,T^*]}\left|\bU_N(t,\cdot) \right|_{H^{s}} \leq M^*\) and $\bU_N([0,T^*]\times\RR^d)\subset\cK^*$ there exists $C^*>0$ depending only on \(s\), \(s'\), \(M^*\) and $\cK^*$ such that for all \(0\leq t\leq T^*\),
		\[\left| \bU_N \right|_{H^{s'}}\leq C^*\left| \bU^0 \right|_{H^{s'}} \exp(C^* M^* t).\]
		
		The same results holds replacing~\eqref{eq.symable-smooth0} with~\eqref{eq.symable-smooth} if additionally Assumption~\ref{assump.AS} holds.
	\end{Proposition}
	\begin{proof}
		As in Propositions~\ref{prop.sym_sharp_num_bound} and~\ref{prop.sym-smoothnum-bound}, the key ingredient is an apriori estimate, and we focus on the derivation of such estimates for solutions \(\bU_N \in \cC^1(I_N; H^\sigma((2\pi\TT)^d)^n)\) for all $\sigma\in\RR$ (recall $\bU_N=\P_N\bU_N$ since $\S_N=\S_N\P_N$). We consider system~\eqref{eq.symable-smooth} which is the most involved, the case of system~\eqref{eq.symable-smooth0} being obtained in the same way. Indeed, while the symmetrizer of the system~\eqref{eq.hyp}, $S(\bU)$, readily offers a suitable symmetrizer of the semi-discretized system~\eqref{eq.symable-smooth0}, for~\eqref{eq.symable-smooth} we need to consider a modification, namely
		\begin{equation}\label{eq.def-tilde-S}
			\tilde{S}(\bU)[\circ] = S^0+\S_N^{1/2} (S^1(\bU) \S_{N}^{1/2}[\circ]).
		\end{equation}
		Recall ${\S_N=\Diag(S_N(D))}$ and the symbol $S_N$ is nonnegative, hence $\S_N^{1/2}=\Diag(S_N(D)^{1/2})$ is well-defined. Let us first prove that for any $\cK\subset \cU$ compact, there exists $0<\alpha\leq\beta<\infty$ such that for all $\bU\subset \cK$  and $\bV\in\RR^n$ one has
		\begin{equation}\label{eq.S-coercive}
			\alpha\left|\bV \right|_{L^2}^2 \leq \big(\tilde{S}(\bU)\bV, \bV\big)_{L^2} \leq \beta\left|\bV \right|_{L^2}^2.
		\end{equation}
		Note first that the result holds for $S(\bU)$ by Assumption~\ref{assump.S1}, and hence for \(S^0=S(\bz)\) since \(\bz \in \cU\). The result then follows from the identities
		\begin{align*}
			\big(\tilde{S}(\bU) \bV, \bV\big)_{L^2} &= (S^0 \bV, \bV)_{L^2} + \big(S^1(\bU) \S_N^{1/2} \bV, \S_N^{1/2} \bV\big)_{L^2}\\
			&= \big(S^0(\Id - \S_N)^{1/2}\bV, (\Id - \S_N)^{1/2}\bV\big)_{L^2} + \big(S(\bU)\S_N^{1/2}\bV, \S_N^{1/2} \bV\big)_{L^2}
			\end{align*}
		and 
		\[\left|\bV \right|_{L^2}^2 = \left|(\Id - \S_N)^{1/2}\bV \right|_{L^2}^2+\left|\S_N^{1/2}\bV \right|_{L^2}^2.\]
		Now we want to show that for any $s'\geq s>1+d/2$ and $\bU_N$ solution to~\eqref{eq.symable-smooth} taking values in $\cK\subset \cU$ compact one has
		\begin{equation}\label{eq.symable_tder_est}
			\frac{1}{2} \frac{\dd}{\dd t}\left(\tilde{S}(\bU_N) \Lambda^{s'} \bU_N, \Lambda^{s'} \bU_N \right)_{L^2} \leq C(\cK,\left|\bU_N \right|_{H^{s}})\left|\bU_N \right|_{H^{s}}\left|  \bU_N\right|_{H^{s'}}^2.
		\end{equation}
		Indeed, if~\eqref{eq.symable_tder_est} holds, applying~\eqref{eq.S-coercive} in~\eqref{eq.symable_tder_est} and using Grönwall's Lemma yields the following estimate
		\begin{align*}\alpha^{1/2}(\cK^*)\left|\bU_N(t,\cdot) \right|_{H^{s'}}&\leq \cF_{s'}(t)\leq \cF_{s'}(0)\exp\left( \alpha^{-1}(\cK^*)\int_0^t C(\cK^*,\left|\bU_N (\tau,\cdot)\right|_{H^{s}}) \left|\bU_N(\tau,\cdot) \right|_{H^{s}}\dd\tau\right)\\
			&\leq \beta^{1/2}(\cK^*)\left|\bU_N(0,\cdot) \right|_{H^{s'}}\exp\left( t\times \alpha^{-1}(\cK^*)  C(\cK^*;\cM_{s}(t))  \cM_{s}(t)\right)
		\end{align*}
		where \(\cK^*\) is compact with \(\cK \subset \cU^* \subset \cK^* \subset \cU\) and \(\cU^*\) open, and where we denote
		\[ \cF_{s'}(\tau)\coloneq\left.\left(\tilde{S}(\bU_N) \Lambda^{s'} \bU_N, \Lambda^{s'}\dot \bU_N \right)_{L^2}^{1/2}\right|_{t=\tau} \quad \text{ and \quad } \cM_{s}(t)\coloneq \sup_{\tau\in[0,t]}\left|\bU_N(\tau,\cdot) \right|_{H^{s}}.\]
		The estimate is valid as long as $\bU_N(\tau,\cdot)$ takes values in $\cK^*$ for all $\tau\in[0,t]$. This is ensured by 
		\[ \left|\bU_N(\tau,\cdot) -\bU_N^0 \right|_{L^\infty} \leq \int_0^t\left|\partial_t\bU_N(\tau,\cdot) \right|_{L^\infty}\dd\tau \leq t\times C(\cK^*;\cM_{s}(t))\cM_{s}(t) \]
		which follows from the system~\eqref{eq.symable-smooth} and continuous Sobolev embedding $H^s((2\pi\TT)^d)^n\subset W^{1,\infty}((2\pi\TT)^d)^n$ (Proposition~\ref{prop.embedding}). By means of these estimates, we can employ the continuity argument as in the proof of Proposition~\ref{prop.sym_sharp_num_bound} to conclude the proof.

		Let us now prove the estimate~\eqref{eq.symable_tder_est}.
		Apply $\Lambda^{s'}$ to the system~\eqref{eq.symable-smooth} and denote $\dot \bU_N \coloneq \Lambda^{s'} \bU_N$ to infer
		\[
			\partial_t \dot \bU_N + \sum_{j=1}^d  (A_j^0 + \S_N (A_j^1(\bU_N)[\circ]))\partial_{x_j}\dot\bU_N=-\S_N\left(\sum_{j=1}^d [\Lambda^{s'},A_j^1(\bU_N)]\partial_{x_j}\bU_N\right).
		\]
		Applying the operator $\tilde{S}(\bU_N)$ to the system and using the self-adjointness of \(\tilde{S}(\bU_N)\) we get
		\begin{align*}\label{eq.estimate_symable_lambda}
			\frac{1}{2} \frac{\dd}{\dd t}\left(\tilde{S}(\bU_N) \dot \bU_N, \dot \bU_N \right)_{L^2} =&\frac{1}{2}\left([\partial_t, \tilde{S}(\bU_N)] \dot \bU_N, \dot \bU_N \right)_{L^2}
			- \sum_{j=1}^d\left( \tilde{S}(\bU_N)  (A_j^0 + \S_N (A_j^1(\bU_N)[\circ]))\partial_{x_j}\dot\bU_N, \dot \bU_N \right)_{L^2}\\
			&	-\sum_{j=1}^d \left( \tilde{S}(\bU_N)\S_N\left([\Lambda^{s'},A_j^1(\bU_N)]\partial_{x_j}\dot\bU_N\right), \dot \bU_N  \right)_{L^2}.
		\end{align*}
		By using that \(\partial_t\) commutes with \(S^0\) and \(\S_N^{1/2}\) we get 
		\begin{equation*}
			\left|\left([\partial_t, \tilde{S}(\bU_N)] \dot \bU_N, \dot \bU_N \right)_{L^2} \right| = \left|\left([\partial_t, S^1(\bU_N)]\S_N^{1/2} \dot \bU_N,  \S_N^{1/2} \dot\bU_N \right)_{L^2}\right| \leq C\left(\cK,\left|\bU_N \right|_{L^\infty} \right)\left|\partial_t \bU_N \right|_{L^\infty} \left| \S_N^{1/2} \dot \bU_N\right|_{L^2}^2. 
		\end{equation*}
		It follows from the system~\eqref{eq.symable-smooth}, continuous Sobolev embedding $H^s((2\pi\TT)^d)^n\subset W^{1,\infty}((2\pi\TT)^d)^n$ (Proposition~\ref{prop.embedding}) and product and composition estimates (Propositions~\ref{prop.product_estimates} and~\ref{prop.composition_estimates}) that
		\begin{align*}
			\left|\left([\partial_t, \tilde{S}(\bU_N)] \dot \bU_N, \dot \bU_N \right)_{L^2} \right| \leq C\left(\cK,\left|\bU_N \right|_{H^{s}}\right)\left|\bU_N \right|_{H^{s}} \left|  \bU_N\right|_{H^{s'}}^2. 
		\end{align*}
		For the third term, we use the boundedness of the operators $\tilde{S}(\bU_N)$ and $\S_N$ and the commutator estimate in Proposition~\ref{prop.commutator_s} and find
		\begin{equation*}
			\left| \left( \tilde{S}(\bU_N)\S_N\left( [\Lambda^{s'},A_j^1(\bU_N)]\partial_{x_j}\bU_N\right), \dot \bU_N  \right)_{L^2}\right| \leq C\left(\cK,\left|\bU_N \right|_{H^{s}}\right)\left|\bU_N \right|_{H^{s}} \left|  \bU_N\right|_{H^{s'}}^2.
		\end{equation*}
		Estimating the remaining term requires more care, since \(\tilde{S}(\bU_N)\) is not a perfect symmetrizer for the system. Observe
		\begin{align*}
			\tilde{S}(\bU_N)(A^0_j + \S_N(A_j^1(\bU_N))) 
			=&  S^0A_j^0 + \S_N^{1/2}(S^1(\bU_N) A_j^0+S^0  A_j^1(\bU_N))\S_N^{1/2} + \S_N S^1(\bU_N)  A_j^1(\bU_N)\S_N\\
			+& \S_N^{1/2} S^0 [\S_N^{1/2}, A_j^1(\bU_N)] 
			 - \S_N^{1/2} [\S_N^{1/2}, S^1(\bU_N)] \S_N A_j^1(\bU_N)+ \S_N S^1(\bU_N) [\S_N, A_j^1(\bU_N)]\\
			 =& {\rm Sym}(\bU_N) + {\rm Com}(\bU_N).
		\end{align*}
		Notice that all terms in the first line ---the sum being denoted ${\rm Sym}(\bU_N)$--- are symmetric operators by Assumption~\ref{assump.AS},  while the remaining terms ---the sum being denoted ${\rm Com}(\bU_N)$--- involve commutators between \(\S_N^{1/2}\) and either \(S^1(\bU_N)\) or \(A_j^1(\bU_N)\). Integration by parts and symmetry considerations implies that 
		\begin{align*}
			\left| \big({\rm Sym}(\bU_N) \partial_{x_j} \dot \bU_N, \dot \bU_N\big)_{L^2} \right| \leq C(\cK,\left|\bU_N \right|_{H^{s}})\left|\bU_N \right|_{H^{s}}\left| \bU_N \right|_{H^{s'}}^2.
		\end{align*}
		For the other terms we use that \(\S_N^{1/2}\) satisfies the commutator estimate in Proposition~\ref{prop.commutator_0} to infer
		\begin{align*}
			\left| \big({\rm Com}(\bU_N) \partial_{x_j} \dot \bU_N, \dot \bU_N\big)_{L^2} \right| \leq C(\cK,\left|\bU_N \right|_{H^{s}})\left|\bU_N \right|_{H^{s}}\left| \bU_N \right|_{H^{s'}}^2.
		\end{align*}
		Combining all the estimates shows~\eqref{eq.symable_tder_est}, which concludes the proof.
	\end{proof}

	Having established uniform bounds for solutions \(\bU_N\) to the semi-discretized system~\eqref{eq.symable-smooth0} and~\eqref{eq.symable-smooth}, we now show the convergence towards corresponding solutions of the underlying continuous problem~\eqref{eq.hyp} as \(N\to \infty\). 

	\begin{Proposition}[Convergence]\label{prop.symable-difference}
		Suppose that the system~\eqref{eq.hyp} satisfies Assumptions~\ref{assump.A1} and~\ref{assump.S1}. Let $s>d/2+1$ and \(\bU^0\in H^s((2\pi\TT)^d)^n\) such that $\bU^0$ takes values into the hyperbolic domain $\cU$. Denote $\bU\in\cC(I;H^s((2\pi\TT)^d)^n)$ the maximal-in-time classical solution to~\eqref{eq.hyp}, and $\bU_N\in\cC(I_N;H^s((2\pi\TT)^d)^n)$ the maximal-in-time classical solution to~\eqref{eq.symable-smooth0}. 
		For every compact subset \(I^* \subset I\), there is an \(N_0\in\NN\) and $C^*>0$, depending only on $s$, $|I^*|$, $\sup_{t\in I^*} \left|\bU(t,\cdot)\right|_{H^s}$ and $\cK^*\subset \cU$ compact such that $\bU(I^*\times\RR^d)\subset \cK^*$ such that for all $N\geq N_0$, one has $I_N\supset I^*$ and for any $0\leq r\leq s$,
		\[\sup_{t\in I^*}\norm{(\bU-\bU_N)(t,\cdot)}_{H^r}\leq C^*\, M^*\, N^{r-s} .\]
		
				The same results holds replacing~\eqref{eq.symable-smooth0} with~\eqref{eq.symable-smooth} if additionally Assumption~\ref{assump.AS} holds.
	\end{Proposition}
	\begin{proof}
		The proof is similar to that of Propositions~\ref{prop.sym-sharpdifference} and~\ref{prop.sym-smoothdifference}, and we only sketch the main arguments. We consider the system~\eqref{eq.symable-smooth0}; system~\eqref{eq.symable-smooth} can be treated in a similar way after introducing the symmetrizer $\tilde{S}(\bU)$ defined in~\eqref{eq.def-tilde-S} as in the proof of Proposition~\ref{prop.symable_smooth_num_bound}. Denote $\bD_N\coloneq\bU-\bU_N$ and
		use that $(\Id-\S_N)\S_{N/2}^{1/2}=0$ to infer
		\begin{multline*}\partial_t \S_{N/2}^{1/2}\bD_N + \sum_{j=1}^d A_j(\bU_N)\partial_{x_j}\S_{N/2}^{1/2}\bD_N+ \sum_{j=1}^d[\S_{N/2}^{1/2}, A_j(\bU_N)]\partial_{x_j}\bD_N\\
				+\S_{N/2}^{1/2}\left(\sum_{j=1}^d (A_j(\bU)-A_j(\bU_N))\partial_{x_j}\bU\right)=\bz.
		\end{multline*}
		Applying the symmetrizer \(S(\bU_N)\) and testing against $\S_{N/2}^{1/2}\bD_N$ yields,
		\begin{align*}
			\frac12\frac{\dd}{\dd t}\left(S(\bU_N)\S_{N/2}^{1/2}\bD_N, \S_{N/2}^{1/2}\bD_N\right)_{L^2}=
			&\frac12\left( [\partial_t, S(\bU_N)] \S_{N/2}^{1/2}\bD_N, \S_{N/2}^{1/2}\bD_N \right)_{L^2} \\
			& - \sum_{j=1}^d \left( S(\bU_N)A_j(\bU_N)\partial_{x_j}\S_{N/2}^{1/2}\bD_N , \S_{N/2}^{1/2}\bD_N\right)_{L^2}\\
			& - \sum_{j=1}^d \left(S(\bU_N) [\S_{N/2}^{1/2}, A_j(\bU_N)]\partial_{x_j}\bD_N, \S_{N/2}^{1/2} \bD_N  \right)_{L^2} \\
			& - \sum_{j=1}^d \left(S(\bU_N)\S_{N/2}^{1/2}\left((A_j(\bU)-A_j(\bU_N))\partial_{x_j}\bU  \right), \S_{N/2}^{1/2} \bD_N\right)_{L^2}.
		\end{align*}
		The first two terms can be estimated in the standard manner using that $S(\cdot)$ is a symmetrizer, {\em i.e.} Assumption~\ref{assump.S1}. The third term is estimated using the regularizing properties of the commutator with \(\S_{N/2}^{1/2}\), Proposition~\ref{prop.commutator_0}. The fourth term is estimated thanks to Assumption~\ref{assump.A1}. Altogether we obtain
		\[\frac12\frac{\dd}{\dd t}\left(S(\bU_N)\S_{N/2}^{1/2}\bD_N, \S_{N/2}^{1/2}\bD_N\right)_{L^2} \leq C(\norm{\bU}_{H^{s}},\norm{\bU_N}_{H^{s}}) \left( \norm{\bU}_{H^s}+\norm{\bU_N}_{H^s} \right)\big(\norm{\S_{N/2}^{1/2}\bD_N}_{L^2}+\norm{\bD_N}_{L^2}\big)\norm{\S_{N/2}^{1/2}\bD_N}_{L^2}.\]
		With this estimate in hand, we can follow the proof of Proposition~\ref{prop.sym-sharpdifference}. Using that $\|\Id-\S_{N/2}^{1/2}\|_{H^s\to L^2}\leq \langle N/4\rangle^{-s}$, the coercivity of $S(\bU_N)$ ---see~\eqref{eq.S-coercive}--- and the uniform estimates for $\bU_N$ stated in Proposition~\ref{prop.symable_smooth_num_bound} we infer (by Grönwall's lemma), denoting $M^*\coloneq 2\sup_{t\in I^*}\left|\bU(t,\cdot)\right|_{H^{s_0}}$ with $d/2+1<s_0<s$ and assuming \(\sup_{t\in I^*}\left|\bU_N(t,\cdot) \right|_{H^{s_0}} \leq M^*\) and $\bU_N(I^*\times\RR^n)\subset \cK^*\subset \cU$,
			\[\norm{\S_{N/2}^{1/2}\bD_N(t,\cdot)}_{L^2}\leq C(\cK^*,M^*) M^*\big(\norm{\S_{N/2}^{1/2}\bD_N\big\vert_{t=0}}_{L^2} +t  N^{-s}\norm{\bU^0}_{H^s}\big) \, \exp\Big(C(\cK^*,M^*)M^* t\Big),\]
		and we have moreover $\S_{N/2}^{1/2}\bD_N\big\vert_{t=0}=0$ and
		\[\norm{\bD_N}_{L^2}\leq \norm{\S_{N/2}^{1/2}\bD_N}_{L^2}+\langle N/4\rangle^{-s}\norm{\bU^0}_{H^s}.\]
		This yields the desired estimate for $r=0$, and the general case follows by interpolation. A continuity argument as in Proposition~\ref{prop.sym-sharpdifference} allows to secure the bound \(\sup_{t\in I^*}\left|\bU_N(t,\cdot) \right|_{H^{s_0}} \leq M^*\) and the assumption ${\bU_N(I^*\times\RR^n)\subset \cK^*\subset \cU}$ for $N$ sufficiently large  (by the convergence $\bU_N(t,\cdot)\to \bU(t,\cdot)$ as $N\to\infty$ in $H^{s_0}(\RR^d)^n\subset W^{1,\infty}(\RR^d)^n$) along the desired estimate, which concludes the proof .
	\end{proof}

	\begin{Remark}\label{rem.general1} 
		Let us comment on the restrictive Assumption~\ref{assump.AS} arising in the study of system~\eqref{eq.symable-smooth}. This assumption allows to ensure that we can construct a symmetrizer operator \(\tilde{S}(\cdot)\) that is 
		\begin{enumerate}
			\item \label{i} bounded, coercive and self-adjoint for the $L^2(\RR^d)^n$ inner-product, and 
			\item \label{ii} such that \(\tilde{S}(\bU)(A_j^0 + \S_N A_j^1(\bU))\) is self-adjoint up to regularizing operators of order $-1$.
		\end{enumerate} 
		Let us notice that if Assumption~\ref{assump.AS} does not hold, it is possible to modify the semi-discretized scheme~\eqref{eq.symable-smooth} in a way that allows for a symmetrizer operator satisfying at least partially the above requirements.
		Specifically, consider the system
		\begin{equation}\label{eq.symable-smooth1}
			\partial_t \bU_N+\sum_{j=1}^d \sum_{\ell=0}^{m_j}  \S_N^\ell (A_j^\ell(\bU_N)[\circ])\partial_{x_j}\bU_N=\bz, \quad \bU_N\vert_{t=0} = \P_N \bU^0,
		\end{equation}
		and associated symmetrizer
		\[\tilde{S}(\bU)[\circ] \coloneq \sum_{\ell=0}^m  \S_N^{\ell/2}(S^\ell(\bU) \S_{N}^{\ell/2}[\circ]),\]
		where we used the convention $\S_N^0=\Id$ and $\S_N^\ell=\underbrace{\S_N\circ \dots \circ \S_N}_{\text{$\ell$ times}}$, and decompositions
		\begin{equation*}
			S(\bU) = \sum_{\ell= 0}^{m} S^\ell(\bU),  \qquad A_j(\bU) = \sum_{\ell=0}^{m_j} A_j^\ell(\bU) \qquad (j\in\{1,\dots,d\})
		\end{equation*}
		where entries of $S^\ell(\cdot)$ and $A_j^\ell(\cdot)$ are homogeneous polynomials of degree $\ell$.
	
		By Taylor expansion about the origin of $S(\cdot)$ and $S(\cdot)A_j(\cdot)$ and homogeneity we find that for all $\bU\in\cU$, $S^\ell(\bU)$ and $\sum_{\ell_1+\ell_2=\ell}S^{\ell_1}(\bU)A_j^{\ell_2}(\bU)$ are symmetric. This shows that $\tilde S(\bU)$ is self-adjoint  for the $L^2(\RR^d)^n$ inner-product and, denoting $\tilde A_j(\bU)[\circ]\coloneq \sum_{\ell=0}^{m_j}  \S_N^\ell (A_j^\ell(\bU)[\circ])$, one has that $\tilde S(\bU)\tilde A_j(\bU)$ is self-adjoint up to regularizing operators of order $-1$. Notice however that some additional restrictions on $\bU\in\cU$ may be necessary to enforce the coercivity of the operator $\tilde S(\bU)$.
	\end{Remark}

	\subsection{Discretization with sharp low-pass filters}
	\label{sec.symable.sharp}
	
	In this section we consider the case of spatial discretization through the sharp low-pass filter $\P_N=\Diag(P_N(D))$ where $P_N(\cdot)=\bo_{\llbracket -N,N \rrbracket^d}(\cdot)$:
	\begin{equation}\label{eq.symable-sharp}
		\partial_t \bU_N+\P_N\left(\sum_{j=1}^d A_j(\bU_N)\partial_{x_j}\bU_N\right)=\bz, \quad \bU_N\vert_{t=0} = \P_N\bU^0.
	\end{equation}
	The analysis of the previous section fails, due to the lack of good commutator properties of the operator $\P_N$. Specifically, when applying a symmetrizer \(S(\cdot)\) to the underlying system~\eqref{eq.hyp} we may write 
	\[		S(\bU_N)\partial_t \bU_N+\sum_{j=1}^d S(\bU_N)A_j(\bU_N)\partial_{x_j}\bU_N=-\sum_{j=1}^d S(\bU_N)(\Id-\P_N)\big(A_j(\bU_N)\partial_{x_j}\bU_N\big).\]
	In order to control the energy functional $\cF_s(\bU_N)\coloneq\left( S(\bU_N)\Lambda^s\bU_N,\Lambda^s\bU_N\right)^{1/2}\approx\norm{\bU_N}_{H^s}$ as in Proposition~\ref{prop.symable_smooth_num_bound}, we wish to control {(uniformly with respect to $N$) the quantity $J_N(\bU_N,\Lambda^s\bU_N)$ where}
	\[{J_N(\bU,\bV) \coloneq \left( \P_N \Big(S(\bU)(\Id-\P_N)\big(A_j(\bU)(\partial_{x_j}\P_N \bV )\big)\Big), \bV\right)_{L^2}.}\] 
	As we shall see in an example in Section~\ref{sec.illustration}, it turns out we cannot improve in general the bound 
	$ J_N=\cO(N \norm{\bV}_{L^2}^2)$. Notice that in the symmetric cases discussed in Section~\ref{sec.sym.sharp}, namely when $S=\Id$, we have $J_N=0$.
	
	In this section we consider symmetrizable systems satisfying the following assumption.
	\begin{Assumption}{A.5}\label{assump.A2}  
		There exists $S(\cdot)$ and an open set \(\cU \subset \RR^{n}\) with \(\bz\in\cU\) 
		such that for all $\bU \in\cU$,  $S(\bU)$ is real-valued, symmetric positive definite, and for all \(j\in \{1, \ldots, d\}\) there exists $S_j^0$ real-valued symmetric matrix with constant coefficients such that for all $\bU\in\cU$, 
		\[ A_j(\bU) = S_j^0 S(\bU). \]
		{We assume that all entries of $S(\cdot)$ are polynomial.}
	\end{Assumption}
	\begin{Remark}\label{R.assumptions}Assumption~\ref{assump.A2} is a special case of symmetrizable systems, as it implies Assumptions~\ref{assump.A1} and~\ref{assump.S1}.
	\end{Remark}
	\begin{Remark}[Hamiltonian systems]
		Assumption~\ref{assump.A2} is motivated by the Hamiltonian structure of the underlying system. Indeed, denote
		\[ \J \coloneq \sum_{j=1}^d S_j^0 \partial_{x_j}\]
		the constant-coefficient skew-symmetric (for the $L^2(\RR^d)^n$ inner-product) operator
		and $\sH:\bU\in\cU\to \RR$ coercive functional such that for all $\bU\in\cU$,
		\[ {\rm Hess}(\sH(\bU)) = S(\bU).\]
		Then we remark that under the Assumption~\ref{assump.A2}~\eqref{eq.hyp} takes the Hamiltonian form
		\[ \partial_t\bU + \J \big( \nabla_{\bU}\sH(\bU)\big)=\bz,\]
		where $\nabla_{\bU}\sH:\cU\to\RR^n$ is the Jacobian of $\sH$. 
		
		Moreover, noticing that~\eqref{eq.symable-sharp} also enjoys a Hamiltonian structure,
		\[ \partial_t\bU_N + \J \big( \nabla_{\bU}\sH_N(\bU_N)\big)=\bz\]
		where	$\sH_N(\bU) = \sH(\P_N\bU)$, we find that ${\rm Hess}(\sH_N(\bU_N))= \P_N S(\bU_N)\P_N$ is a symmetrizer of the system~\eqref{eq.symable-sharp}.
	\end{Remark}
	Under this assumption, we have the following bound on solutions \(\bU_N\) to the semi-discretized problems~\eqref{eq.symable-sharp}.
	\begin{Proposition}[Uniform estimates]
	Under the Assumption~\ref{assump.A2}, the statement of Proposition~\ref{prop.symable_smooth_num_bound} holds replacing~\eqref{eq.symable-smooth0} with~\eqref{eq.symable-sharp}.
	\end{Proposition}
	\begin{proof}
		We follow very closely the proof of Proposition~\ref{prop.symable_smooth_num_bound}, and only sketch how the necessary estimates can be obtained. Apply $\P_N S(\bU_N)\P_N\Lambda^{s'}$ to the system~\eqref{eq.symable-sharp} and use the identity $A_j(\bU_N)=S_j^0S(\bU_N)$  to infer
		\begin{multline*}
			\P_N \big(S(\bU_N)\P_N (\partial_t \Lambda^{s'} \bU_N)\big) + \sum_{j=1}^d \P_N \Big(S(\bU_N)\big(\P_N  S_j^0 \P_N(S(\bU_N) \partial_{x_j}\Lambda^{s'}\bU_N)\big)\Big)\\
			=-\P_N\left(\sum_{j=1}^d  S(\bU_N)\P_N\big([\Lambda^{s'},S_j^0S(\bU_N)]\partial_{x_j}\bU_N\big)\right)
		\end{multline*}
		where we used that $\P_N^2=\P_N$ commutes with $S_j^0$ and $\Lambda^{s'}$. We can now test the identity against $\Lambda^{s'} \bU_N$ and use the self-adjointness of \(\P_N\), \(S_{j}^0\) and  \(S(\bU_N)\) as well as the identity $\bU_N=\P_N\bU_N$ to infer
		\begin{align*}
			\frac{1}{2} \frac{\dd}{\dd t}\left(S(\bU_N) \Lambda^{s'} \bU_N, \Lambda^{s'} \bU_N \right)_{L^2} =&\frac{1}{2}\left([\partial_t, S(\bU_N)] \Lambda^{s'} \bU_N, \Lambda^{s'} \bU_N 	\right)_{L^2}\\
			&+ \sum_{j=1}^d \Big(  S_j^0 \P_N\big([\partial_{x_j},S(\bU_N)] \Lambda^{s'}\bU_N\big),\P_N\big(S(\bU_N)\Lambda^{s'}\bU_N\big)\Big)_{L^2}\\
			&-	\sum_{j=1}^d  \left(S(\bU_N)\P_N\big([\Lambda^{s'},S_j^0S(\bU_N)]\partial_{x_j}\bU_N\big) ,\Lambda^{s'}\bU_N\right)_{L^2}.
		\end{align*}
		We then proceed as in Proposition~\ref{prop.symable_smooth_num_bound} and obtain the energy estimate valid as long as $\bU_N$ takes values into $\cK\subset \cU$ compact:
		\[
			\frac{1}{2} \frac{\dd}{\dd t}\left(S(\bU_N) \Lambda^{s'} \bU_N, \Lambda^{s'} \bU_N \right)_{L^2} \leq C\left|\bU_N \right|_{H^{s}}\left|  \bU_N\right|_{H^{s'}}^2,
		\]
		where the constant $C$ depends only on $s,s',\cK$  and non-decreasingly on $\left|\bU_N \right|_{H^{s}}$. We also have immediately the coercivity of $S(\bU)$: for any $\bU\subset \cK$  and $\bV\in\RR^n$ one has
		\[\alpha\left|\bV \right|_{L^2}^2 \leq (S(\bU)\bV, \bV) \leq \beta\left|\bV \right|_{L^2}^2,\]
		where $0<\alpha\leq\beta<\infty$ depend uniquely on $\cK$. These two ingredients yield the desired result.
	\end{proof}

	Having established uniform bounds for solutions \(\bU_N\) to the semi-discretized system~\eqref{eq.symable-sharp}, we infer the convergence towards corresponding solutions of the underlying continuous problem~\eqref{eq.hyp} as \(N\to \infty\). 
	\begin{Proposition}[Convergence]\label{prop.symable-sharpdifference}
		Under Assumption~\ref{assump.A2}, the statement of Proposition~\ref{prop.symable-difference} holds replacing~\eqref{eq.symable-smooth0} with~\eqref{eq.symable-sharp}.
	\end{Proposition}
	\begin{proof}
		The proof is identical to that of Proposition~\ref{prop.symable-difference}.
	\end{proof}
	
	\section{Numerical experiments for the Saint-Venant system}
	\label{sec.illustration}
	
	We shall illustrate our findings and investigate numerically the standard Saint-Venant (or shallow water) system

	\begin{equation}\label{eq.SV}
		\left\{\begin{array}{l}
			\partial_t\eta+\nabla\cdot\big((1+\eta)\bu\big)=0,\\[1ex]
			\partial_t\bu +\nabla \eta + (\bu\cdot\nabla)\bu=\bz,
		\end{array}\right.
	\end{equation}
 	which describes the propagation of shallow water waves in the flat-bottom situation; see~\cite{Lannes}. Specifically, the scalar variable $\eta$ describes the elevation of the surface of a layer of homogeneous, incompressible and inviscid fluid and the variable $\bu$ represents the layer-averaged horizontal velocity of fluid particles (both depending on time and horizontal space). The gravitational constant and reference depth have been set to $g=1$ and $H=1$.
		 
	It will be interesting to consider the following variant (when $d=2$)
	\begin{equation}\label{eq.SV-2}
		\left\{\begin{array}{l}
			\partial_t\eta+\nabla\cdot\big((1+\eta)\bu\big)=0,\\[1ex]
			\partial_t\bu +\nabla \eta + \tfrac12\nabla(|\bu|^2)=\bz.
		\end{array}\right.
	\end{equation}
	While the two systems are identical when \(d=1\), only the second has a Hamiltonian structure when \(d=2\); see below. As exhibited in the following section, the hyperbolicity domain of the Hamiltonian system~\eqref{eq.SV-2} is a strict subset of the  hyperbolicity domain of the standard system~\eqref{eq.SV}. 
	
	Thanks to these features, numerical experiments on the Saint-Venant systems allow to showcase our numerical findings, that we summarize here for the sake of readability. 
		\begin{itemize}
			\item Numerical experiments validate our results concerning the spectral convergence of the (semi-)discretized solutions as $N\to\infty$, when Assumptions~\ref{assump.A1},~\ref{assump.S1} and~\ref{assump.AS} hold and smooth low-pass filters are used.
			\item In the case of sharp low-pass filters, we have {\em not} been able to observe numerical instabilities when Assumptions~\ref{assump.A1},~\ref{assump.S1} and~\ref{assump.AS} hold but Assumption~\ref{assump.A2} fails. 
			\item Contrarily to sharp low-pass filters, smooth low-pass filters are able to instate a form of stability even outside the domain of hyperbolicity, that is when Assumption~\ref{assump.S1} fails.
		\end{itemize}

	\subsection{Analysis of the Saint-Venant system}
	\label{sec.SV_analysis}
	
	We can apply the analysis of the previous section to systems~\eqref{eq.SV} and~\eqref{eq.SV-2} due to the following result.
	\begin{Lemma}\label{lem.SV}	
		System~\eqref{eq.SV} is a symmetrizable hyperbolic system in the sense of Assumptions~\ref{assump.A1} and~\ref{assump.S1}  with hyperbolic domain $\cU\coloneq \{(\eta,\bu)\in\RR^{1+d} \ : \ 1+\eta >0\}$ and
		\begin{equation}\label{eq.SV-sym}
			S((\eta,\bu)) = \begin{pmatrix}
				1&\bz^\top\\\bz&(1+\eta)\Id
			\end{pmatrix}
		\end{equation} 
		where $\Id$ is the identity matrix in $\RR^d$. Moreover, the additional Assumption~\ref{assump.AS} holds (see Remark~\ref{rem.AS-linear}).

		System~\eqref{eq.SV-2} satisfies Assumption~\ref{assump.A2} (and hence Assumptions~\ref{assump.A1} and~\ref{assump.S1}; see Remark~\ref{R.assumptions}) with hyperbolic domain  \(\cU_{\sH} \coloneq \{(\eta,\bu)\in\RR^{1+d} \ : \ 1 + \eta - \left|\bu \right|^2 >0\}\) and
		\begin{equation}\label{eq.SV-sym-H}
			S_{\sH}((\eta, \bu)) = \begin{pmatrix} 1 &\bu^\top \\ \bu &(1+\eta)\Id \end{pmatrix} \text{ and } S_j^0 = \begin{pmatrix} 0 & \be_j^\top \\ \be_j & \bz \end{pmatrix} \text{ where } \be_1 = \begin{pmatrix} 1 \\ 0 \end{pmatrix}, \be_2 = \begin{pmatrix} 0 \\ 1 \end{pmatrix}
		\end{equation}
		(when \(d=2\), set \(\be_j = 1\) for the analogous definitions for \(d=1\))
		and Hamiltonian energy
		\[\sH((\eta,\bu)) = \frac12\int_{(2\pi\TT)^d} \eta^2+(1+\eta)|\bu|^2\dd \bx.\]
	\end{Lemma}
	\begin{Remark}
		As aforementioned, systems~\eqref{eq.SV} and~\eqref{eq.SV-2} are identical when $d=1$, and hence enjoy both properties. Notice that, when $d=2$, the domain of hyperbolicity of system~\eqref{eq.SV-2}, $\cU_{\sH}$, is strictly embedded in the domain of hyperbolicity of system~\eqref{eq.SV}, $\cU$, while only the former satisfies Assumption~\ref{assump.A2}, associated with its Hamiltonian formulation.
	\end{Remark}
	\begin{proof}
		The systems~\eqref{eq.SV} and~\eqref{eq.SV-2} can be reformulated as 
		\begin{equation}\label{eq.SV-matrix-formulation}
			\partial_t \bU+\sum_{j=1}^d A_j(\bU)\partial_{x_j}\bU=\bz,
		\end{equation}
		with \(\bU = (\eta, \bu)\) and 
		\begin{equation*}
			A_j((\eta, \bu)) = \begin{pmatrix} u_j & (1+ \eta)\be_j^\top \\ \be_j &u_j \Id \end{pmatrix}
		\end{equation*}
		for system~\eqref{eq.SV} and 
		\begin{equation*}
			A_j((\eta, \bu)) = \begin{pmatrix} u_j & (1+ \eta)\be_j^\top \\\be_j & \be_j \bu^\top \end{pmatrix}
		\end{equation*}
		for system~\eqref{eq.SV-2}. It is then straightforward to check the assumptions.
	\end{proof}

	Recall the spatial discretization of the system~\eqref{eq.SV-matrix-formulation} with the smooth low-pass filter \(\S_N\), 
	\begin{equation}\label{eq.SV-smooth0}
		\partial_t \bU_N+\S_N\left(\sum_{j=1}^d A_j(\bU_N)\partial_{x_j}\bU_N\right)=\bz, \quad \bU_N\vert_{t=0} = \P_N\bU^0,
	\end{equation}
	or
	\begin{equation}\label{eq.SV-smooth}
		\partial_t \bU_N+\sum_{j=1}^d (A_j^0 + \S_N (A_j^1(\bU_N)[\circ]))\partial_{x_j}\bU_N=\bz, \quad \bU_N\vert_{t=0} = \P_N \bU^0,
	\end{equation}
	and the spatial discretization with the sharp low-pass filter \(\P_N\), 
	\begin{equation}\label{eq.SV-sharp}
		\partial_t \bU_N+\P_N\left(\sum_{j=1}^d A_j(\bU_N)\partial_{x_j}\bU_N\right)=\bz, \quad \bU_N\vert_{t=0} = \P_N\bU^0.
	\end{equation}

	It follows immediately from Lemma~\ref{lem.SV} that we have convergence of the numerical scheme in all three cases. 

	\begin{Proposition}(Convergence)
		The statement of Proposition~\ref{prop.symable-difference}, concerning {spectral} convergence of solutions to the semi-discrete systems~\eqref{eq.SV-smooth0} and~\eqref{eq.SV-smooth}, holds for the Saint-Venant system~\eqref{eq.SV} whenever $\bU^0$ takes values in ${\cU\coloneq  \{(\eta,\bu)\in\RR^{1+d} \ : \ 1+\eta >0\}}$. If additionally \(\bU^0\) takes values in  \(\cU_{\sH} \coloneq \{(\eta,\bu)\in\RR^{1+d} \ : \ 1 + \eta - \left|\bu \right|^2 >0\}\), then Proposition~\ref{prop.symable-difference} also holds for the Hamiltonian Saint-Venant system~\eqref{eq.SV-2}.

		Whenever \(\bU^0\) takes values in  \(\cU_{\sH}\), we furthermore have that the statement of Proposition~\ref{prop.symable-sharpdifference}, concerning {spectral} convergence of solutions to the semi-discrete system~\eqref{eq.SV-sharp}, holds for the system~\eqref{eq.SV-2}.
	\end{Proposition}
	\begin{proof}
		Lemma~\ref{lem.SV} ensures that the systems~\eqref{eq.SV} and~\eqref{eq.SV-2} satisfy the assumptions of Proposition~\ref{prop.symable-difference} with symmetrizers \(S(\bU)\) and \(S_{\sH}(\bU)\) respectively, and that~\eqref{eq.SV-2} additionally satisfies the assumptions of Proposition~\ref{prop.symable-sharpdifference}. The  domains \(\cU\) and \(\cU_{\sH}\) correspond respectively to the domains for which the symmetrizers \(S(\bU)\) and \(S_{\sH}(\bU)\) in~\eqref{eq.SV-sym},~\eqref{eq.SV-sym-H} are positive definite, as required in Assumptions~\ref{assump.S1} and~\ref{assump.A2}.
	\end{proof}

	\begin{Remark}\label{R.symmetrizer-issue}
		As discussed in the previous section, Section~\ref{sec.symable.sharp}, we require more stringent structural assumptions to show convergence for symmetrizable systems when discretizing with the sharp low-pass filter \(\P_N\). The Saint-Venant system when \(d=1\) illustrates that we may also have to impose more stringent restrictions on the inital data (namely \(\bU^0\) taking values in \(\cU_{\sH}\)), even for systems that satify the structural assumptions. 

		This is because \emph{any} symmetrizer for the underlying system in the sense of Assumption~\ref{assump.S1} satisfying the additional compatibility Assumption~\ref{assump.AS} can be used to construct a symmetrizer to the semi-discrete systems~\eqref{eq.SV-smooth0} and~\eqref{eq.SV-smooth}.
		For the semi-discretization with \(\P_N\)~\eqref{eq.SV-sharp} on the other hand, we use the symmetrizer directly related to the structure of the system through Assumption~\ref{assump.A2}. 
	\end{Remark}
		 
	Let us illustrate the discussion in Remark~\ref{R.symmetrizer-issue}.
	As discussed in the beginning of Section~\ref{sec.symable.sharp}, considering a semi-discrete system with sharp low-pass filter emanating from a symmetrizable continuous system with symmetrizer \(S(\bU)\), one wishes to control  the energy functional \(\cF_s(\bU_N)= {\left( S(\bU_N)\Lambda^s\bU_N,\Lambda^s\bU_N\right)^{1/2}\approx\norm{\bU_N}_{H^s}}\), which in turn requires to control (uniformly with respect to \(N\)) the quantity $J_N(\bU_N,\Lambda^s\bU_N)$ where
	\[{J_N(\bU,\bV) \coloneq \left( \P_N \Big(S(\bU)(\Id-\P_N)\big(A_j(\bU)(\partial_{x_j}\P_N \bV )\big)\Big) , \bV\right)_{L^2}.}\]  
	In the specific case of the Saint-Venant system~\eqref{eq.SV} when \(d=1\), one has
	\[ \cF_s(\bU) = \int_{(2\pi\TT)} (\Lambda^s \eta)^2 + (1 + \eta)(\Lambda^s u)^2 \dd x\]
	and 
	\[	S(\bU)=\begin{pmatrix}
		1&0\\0&1+\eta
	\end{pmatrix}, \quad A(\bU)=\begin{pmatrix}
		u&1+\eta\\1&u
	\end{pmatrix}, \quad \bU=\begin{pmatrix}\eta \\ u\end{pmatrix}.\] 
	Let $\bU\coloneq (\eta_p,u_p)$ where $\eta_p(x) \coloneq -\frac{1}{2}\cos(px)$, $u_p(x)\coloneq\sin(px)$ and $\bV_N\coloneq(0,v_{N,q})$ where ${v_{N,q}(x)\coloneqq\sin((N-q)x)}$ with $0\leq q<p \ll N$. A direct calculation yields
	\begin{align*}
		(\Id-\P_N)\big(A(\bU)(\partial_{x}\P_N \bV_N )\big) &= \frac{N-q}{4}\begin{pmatrix} -\cos((N-q + p)x) \\ 2 \sin((N-q+p)x)  \end{pmatrix},
	\end{align*}
	so that
	\begin{equation*}
		\P_N \Big(S(\bU)(\Id-\P_N)\big(A(\bU)(\partial_{x_j}\P_N \bV_N )\big)\Big) = \frac{N-q}{8 }\begin{pmatrix} 0 \\ -\sin((N-q)x) \end{pmatrix} , 
	\end{equation*}
	and hence
	\begin{equation*}
		J_N(\bU,\bV_N) = -\frac\pi8(N-q).
	\end{equation*}
	This shows that one cannot propagate for positive time (at least in a direct manner) a uniform-in-$N$ control of the energy functional \(\cF_s(\bU_N)\) 
	for $\bU_N$ the solution emerging from initial data \(\bU_N^0\coloneq\bU+\bV_{N}/(N-q)^s\), despite the fact that $\bU_N^0\in\cU$ since \(1 + \eta_N^{0} \geq 1/2 >0\), and \(\left|\bU_N^0 \right|_{H^s} \approx 1\). 
	Notice also that \(\bU_N^0\notin \cU_{\sH} \) since \(  1 + \eta_{N}^{0}(\frac{\pi}{2p}) - |u_{N}^{0}(\frac{\pi}{2p})|^2 =0\), but one could enforce $\bU_N^0\in \cU_{\sH}$ while keeping valid all previous statements by considering {\em e.g.} $u_p(x)=\frac12\sin(px)$. In that case, the Hamiltonian structure allows to propagate the functional $\cF_{\sH,s}(\bU)\approx\norm{\bU_N}_{H^s}$ with  
	\[\cF_{\sH,s}(\bU)\coloneq \left(  S_\sH(\bU)\Lambda^s\bU,\Lambda^s\bU\right)^{1/2}= \int_{(2\pi\TT)} (\Lambda^s \eta)^2 + (1 + \eta)(\Lambda^s u)^2 +2u(\Lambda^s \eta)(\Lambda^s \eta)\dd x.\]

	\subsection{Numerical experiments in dimension one}\label{sec.SV-num-1}
	We seek numerical approximations to~\eqref{eq.SV} with $d=1$ (or, equivalently,~\eqref{eq.SV-2}), \(\eta_M, u_M\), in terms of finite Fourier sums of the form 
	\begin{equation*}
		f(x) = \sum_{k = -M+ 1}^M a_k \exp(\i k x),
	\end{equation*} 
	and similarly for \(u_M\).	The vectors \(\bm{\eta} = (\eta_M(x_1), \ldots, \eta_M(x_{2M}))\), \(\bu = (u_M(x_1), \ldots, u_M(x_{2M}))\) contain the values of \(\eta_M, u_M\) at regularly spaced collocation points \(x_n = -\pi + \pi n/M, n = 1, \ldots, 2M\). We use the discrete Fourier Transform, computed efficiently with a Fast Fourier transform (FFT), to find 
	\begin{equation*}
		\eta_M(x) = \sum_{k = -M+ 1}^M \hat{\bm{\eta}}_k \exp(\i k x), \quad u_M(x) = \sum_{k = -M+ 1}^M \hat{\bu}_k \exp(\i k x), 
	\end{equation*} 
	where \(\hat{\bm{\eta}}= (\hat{\bm{\eta}}_{-M+ 1}, \ldots,  \hat{\bm{\eta}}_{M}), \hat{\bu} = (\hat{\bu}_{-M+1}, \ldots, \hat{\bu}_{M})\) are the coefficients of the Fast Fourier transform of \(\bm{\eta}, \bu\), and \(\hat{\bm{k}} = ( k/\pi\ : \ k = -M+1, \ldots, M)\) are the discrete Fourier modes. Abusing notation, we will incorrectly refer to \(\hat{\bm{\eta}}_k, \hat{\bu}_k\) as Fourier coefficients (they are related to the coefficients \(c_j\) of infinite Fourier series $		f(x) = \sum_{j\in \ZZ} c_j \exp(\i j x)$ through \(\hat{\bm{f}}_k = \sum_{j\in \ZZ} c_{k + 2jM}\)). For functions \(f\in H^s(2\pi\TT)\), the error due to this aliasing effect is of order \(\cO(M^{-s})\). 
	
	Spatial differentiation is now obtained by multiplying the Fourier coefficients with \(\i k\). Nonlinear operations are computed pointwise on collocation points \(x_n\), via inverse Fast Fourier transform. This procedure leads, in general, to aliasing errors. For polynomial nonlinearities (such as for the Saint-Venant system), one can use so-called dealiasing techniques to remove these errors. For quadratic nonlinearities, one may for example use Orszag's \(3/2\)-rule~\cite{Orszag71}, which consists in adding a sufficient number of Fourier modes with coefficients set to zero. For more information on spectral methods and dealiasing techniques, we refer to~\cite{CanutoHussainiQuarteroniEtAl06} and~\cite{Trefethen00}.

	In our numerical codes, to remove aliasing errors from the nonlinear terms while still working with vectors \(\hat{\bm\eta}, \hat{\bu}\) of fixed length, we shall set the highest \(1/3\) of the Fourier modes to zero. As we numerically compute approximate solutions of the semi-discretized equations~\eqref{eq.SV-smooth} and~\eqref{eq.SV-sharp}, this procedure is naturally performed when applying sharp or smooth low-pass filters, \(\P_{N}, \S_{N}\), with $N<2M/3$. For the smooth low-pass filter, we use the example from the introduction, that is \(\S_N = \Diag(S_N(D))\) with \(S_N(\cdot) = S(\cdot/N)\) and \(S(\cdot) = \max(0,\min(1, 2- 2 \left|\cdot \right|))^2\). For the discretization with the smooth low-pass filter, we shall consider only the version~\eqref{eq.SV-smooth} where the low-pass filter is only applied to nonlinear terms. 

	This procedure of semi-discretization in space yields a system of differential equations in time for the Fourier coefficients \(\hat{\bm{\eta}}, \hat{\bu}\). We approximately solve this initial-value problem using an explicit Runge-Kutta 4 method. All numerical simulations are made using the Julia package WaterWaves1D~\cite{DucheneNavaro} and can be reproduced using the scripts available at
	\url{WaterWaves1D.jl/examples/StudySaintVenant.jl}.

	From now on, we denote the number of collocation points by \(2M\), and let \(N=\lfloor 2M/3 \rfloor\), that is, the greatest integer smaller than \(2M/3\). Our numerical scheme maintains the highest \(1/3\) of the Fourier modes to zero at each time-step, see the discussion above. Abusing notation, we will refer to the fully-discretized numerical solution as \(\bU_N = (\eta_N, u_N)\), since only \(2N\) Fourier modes are nonzero. This convention means that \(N\) plays the same role in this section as in the previous, analytical sections. We will compute the numerical solution with \(2M= 2^j, j = 6,\ldots, 15\) collocation points and use time step \(dt = 10^{-5}\). The time step is an order of magnitude smaller than needed to avoid stability issues, and small enough to ensure the error due to the spatial discretization dominates. 
	We will use the solution computed with \(2M=2^{15}\) and sharp low-pass filter as a reference solution \(\bU_\r = (\eta_\r, u_\r)\), and compute the relative error of the numerical solutions \(\bU_N = (\eta_N, u_N)\) by comparing with the reference solution: 
	\[E_s(\bU_N) = \frac{\left|\bU_N - \bU_\r \right|_{H^s}}{\left|\bU_\r \right|_{H^s}}.\]
 	The norms will be computed {approximately} using the Fourier coefficients of the numerical solutions.

	We solve numerically the Saint-Venant system~\eqref{lem.SV} in one spatial dimension. For the tested initial data in \(\cU_{\sH}\), numerical results are in agreement with the analysis. To study the experimental convergence, we consider the following initial data for \(\alpha >0\),
	\begin{equation}\label{eq.init1}
			\eta^{0}(x) = \frac{1}{2}\exp(-\left|x \right|^{\alpha})\exp(-4x^2), \qquad	u^{0}(x) = 0.
	\end{equation}
	Notice \(\bU^0 = (\eta^{0}, u^{0})\) satisfies both \(1 + \eta^{0} >0\) and the stricter condition \(1+ \eta^{0} - (u^{0})^2>0\). The initial surface is a heap of water situated at the origin. Both \(\eta^{0}, u^{0}\) decay to machine precision near \(-\pi, \pi\) and can therefore be seen as periodic. Moreover, \(\bU^0\in H^{\alpha + 1/2}(2\pi\TT)^2\). We let \(\alpha= 1.5\) and simulate the time-evolution up to a final time \(T= 0.5\) with either sharp or smooth lowpass filters applied to the nonlinear terms. A plot of the initial data as well as the decay of its Fourier coefficients (through $E^s(\bU_N)\vert_{t=0}=\frac{\left|(\Id-\P_N)\bU_\r\vert_{t=0}\right|_{H^s}}{\left|\bU_\r\vert_{t=0} \right|_{H^s}}$) is shown in Figure~\ref{fig.init1}. 

	\begin{figure}[tbp]
		\centering
		\subfloat[Plot of the initial data~\eqref{eq.init1}, where \(\eta^{0}\) is the initial surface profile, and \(u^{0}\) is the initial velocity.]{\label{fig:a}\includegraphics[width=0.45\linewidth]{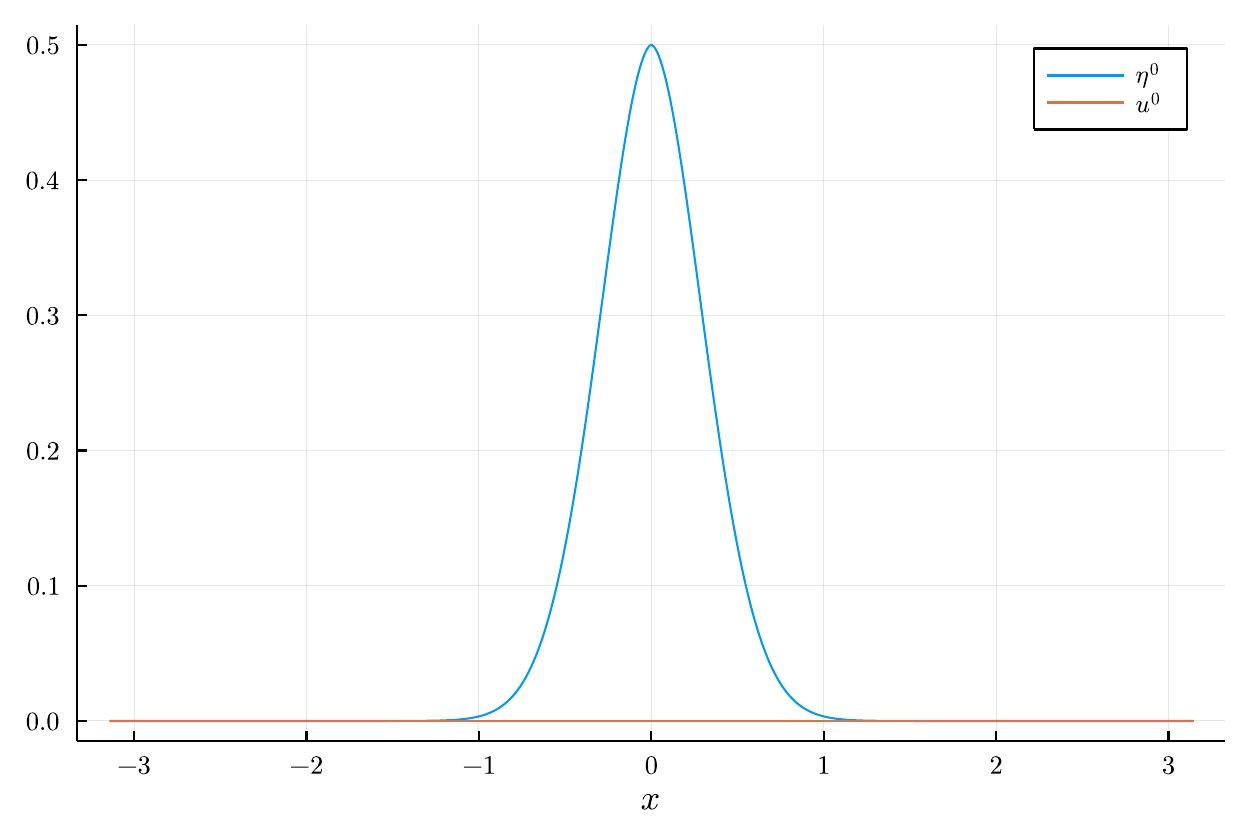}}\qquad
		\subfloat[Decay of the Fourier coefficients for the initial data~\eqref{eq.init1}. The blue and orange points show \(E_s(\bU_N)\vert_{t=0}\) for \(s=0, 1\) respectively. To illustrate, the blue and orange lines have slopes \(-2\) and \(-1\) respectively.]{\label{fig:b}\includegraphics[width=0.45\linewidth]{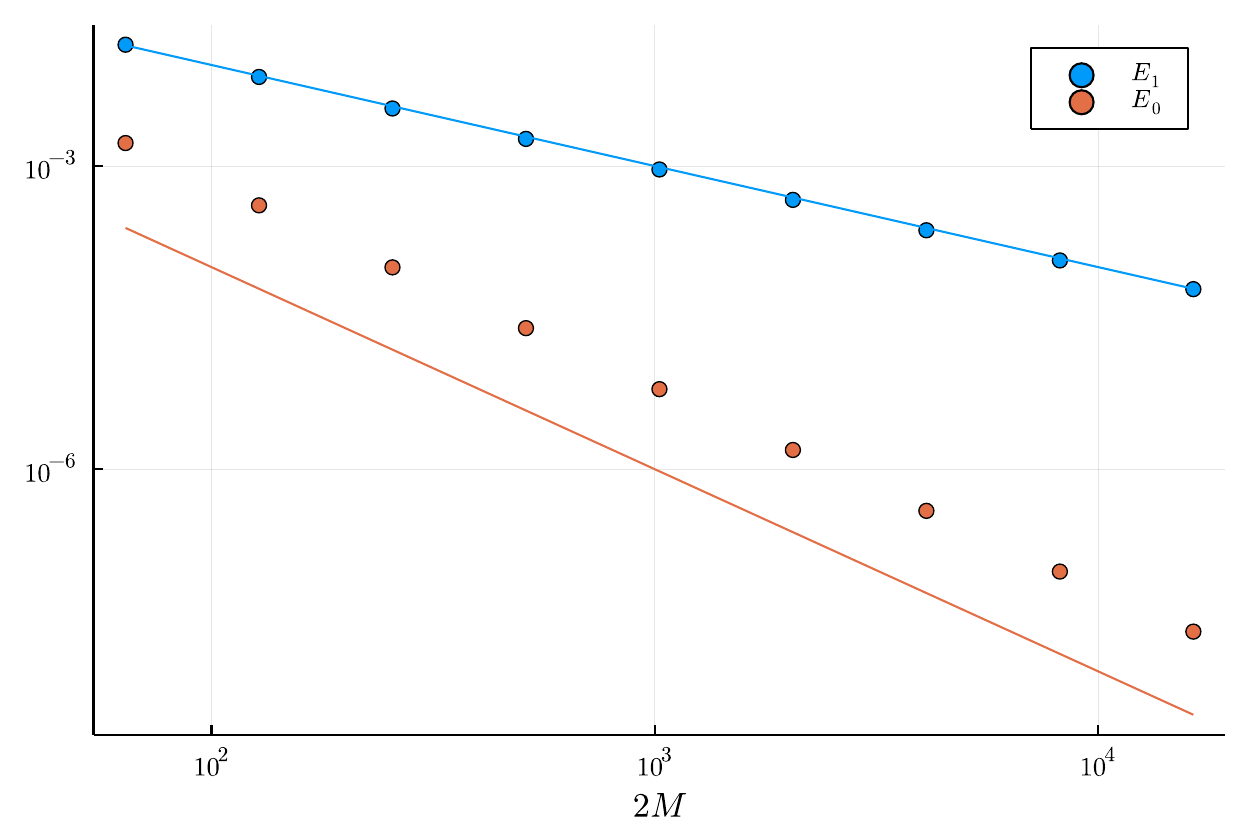}}\\
		\caption{Experiments with initial data~\eqref{eq.init1}.}
		\label{fig.init1}
	\end{figure}

	Figure~\ref{fig.num1} show log-log plots of the error \(E_s\) at time \(T= 0.5\) for the numerical solution with \(2M = 2^j\) where \(j= 6, \ldots, 14\) computed using sharp and smooth low-pass filters. Figure~\ref{fig.num1} shows the relative error measured in the \(L^2\)-norm, \(E_0(\bU_N)\) and in the \(H^1\)-norm, \(E_1(\bU_N)\). 
	
	\begin{figure}[tbp]
		\centering
		\subfloat[Plot illustrating the convergence of the numerical schemes~\eqref{eq.SV-smooth} and~\eqref{eq.SV-sharp} as the number of collocation points \(2M\) increases. The plot shows the relative error of the numerical solution for initial data~\eqref{eq.init1} in \(H^2(2\pi\TT)^2\) measured in the \(L^2\)-norm, \(E_0\) and in the \(H^1\)-norm, \(E_1\) for \(2M = 2^j, j= 6,\ldots, 14\) when using either sharp or smooth low-pass filters. To illustrate, the blue and orange lines have slopes \(-2\) and \(-1\) respectively.  The numerical scheme exhibits spectral convergence with both sharp and smooth low-pass filters.]{\label{fig.num1}\includegraphics[width=0.45\linewidth]{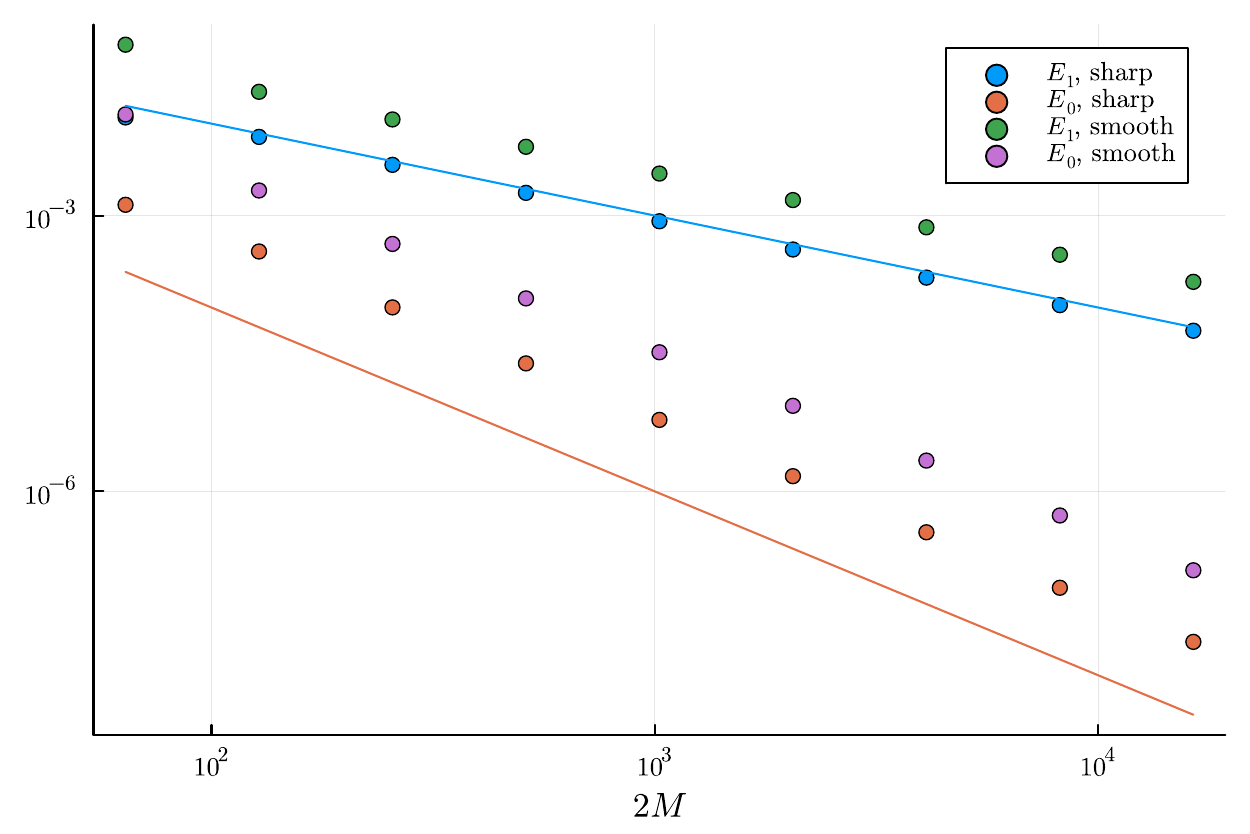}}\qquad
		\subfloat[Experimental order of convergence for the numerical solution with initial data~\eqref{eq.init1}  for both sharp and smooth low-pass filters. The Experimental order of convergence is measured in the \(L^2\)-norm, \(EOC_0\), and in the \(H^1\)-norm, \(EOC_1\).]{\small \begin{tabular}{|c|c|c|c|c|}
			\hline
			 & \multicolumn{2}{|>{\columncolor{lightgray}}c|}{Sharp low-pass filter} & \multicolumn{2}{|>{\columncolor{lightgray}}c|}{Smooth low-pass filter} \\
			\hline
			\(2M\)& \(EOC_0\) & \(EOC_1\) & \(EOC_0\) & \(EOC_1\) \\
			\hline
			\(2^6\) & \(1.69\) & \(0.7\) & \(2.75\) & \(1.7\) \\
			\hline
			\(2^7\) & \(2.02\) & \(1.01\) & \(1.93\) & \(1.00\) \\
			\hline
			\(2^8\) & \(2.02\) & \(1.01\) & \(1.97\) & \(0.99\) \\
			\hline
			\(2^9\) & \(2.04\) & \(1.03\) & \(1.95\) & \(0.97\) \\
			\hline
			\(2^{10}\) &  \(2.03\) & \(1.02\) & \(1.94\) & \(0.96\) \\
			\hline
			\(2^{11}\) & \(2.03\) & \(1.02\) & \(1.98\) & \(0.98\) \\
			\hline
			\(2^{12}\) & \(2.01\) & \(0.99\) & \(1.98\) & \(0.99\) \\
			\hline
			\(2^{13}\) & \(1.95\) & \(0.92\) & \(1.98\) & \(0.98\) \\
			\hline
		\end{tabular}\label{tab.num1}
		\newline \newline}
		\caption{Experiments with initial data~\eqref{eq.init1}.}
	\end{figure}
	
	The experimental order of convergence, given by
	\begin{equation*}
		EOC_{s} = \frac{\log(E_s(\bU_N)/E_s(\bU_{2N}))}{\log(2)},
	\end{equation*}
	is given in Figure~\ref{tab.num1}. With the initial data in \(H^2(2\pi\TT)^d\), we expect from our analysis that the \(L^2\)-error should decay as \(\cO(N^{-2})\) and the \(H^1\)-error should decay as \(\cO(N^{-1})\) when using both the sharp and smooth low-pass filter. This aligns with our numerical results. The absolute error is slightly larger when using the smooth low-pass filter, which is to be expected since applying \(\S_N\) removes more information at each time step than does \(\P_N\). Taking as initial data~\eqref{eq.init1} with other values of \(\alpha\) (we have tested \(\alpha=1,2.5,3\)) also yields the expected results.

Let us now consider initial data $(\eta^{0},u^{0})\in\cU\setminus\cU_\sH$, that is satisfying \(1 + \eta^{0}>0\), but \emph{not} the stricter condition \(1+ \eta^{0} - (u^{0})^2>0\). In particular, we will consider the initial data from the example at the end of Section~\ref{sec.SV_analysis}, with \(q=0,p=1, s= 2\):
\begin{equation}\label{eq.init2}
	\bU_N^0=(\eta^{0},u_N^{0}), \quad \eta_N^{0}(x)=-\frac{1}{2}\cos(x), \quad u_N^{0}(x)=\sin(x)+\frac{\sin(Nx)}{N^2}
\end{equation}
Notice \(f(x) = 1- \frac{1}{2}\cos(x) - \sin^2(x)<0\) for \(x\in (-\frac{\pi}{2}, -\frac{\pi}{3})\cup(\frac{\pi}{3},\frac{\pi}{2})\). On the other hand, \(1 + \eta^{0}(x)>0\) for all \(x\in 2\pi\TT\). We therefore expect the numerical scheme to converge when using the smooth low-pass filter, but that  instabilities may emerge when discretizing with the sharp low-pass filter. Convergence plots for the numerical solutions in both cases are shown in Figure~\ref{fig.num2}. In line with the analysis, we have convergence when using the smooth low-pass filter. However, this is also true for the sharp low-pass filter. {\em Despite many attempts, we have not been able to observe numerical instabilities for initial data we have tested satisfying \(1 + \eta^{0}>0\) but violating \(1+ \eta^{0} - (u^{0})^2>0\). }

\begin{figure}[tbp]
	\centering
	\subfloat[Plot of the initial data~\eqref{eq.init2}, where \(\eta^{0}\) is the initial surface profile, and \(u^{0}\) is the initial velocity.]{\label{fig:init2}\includegraphics[width=0.45\linewidth]{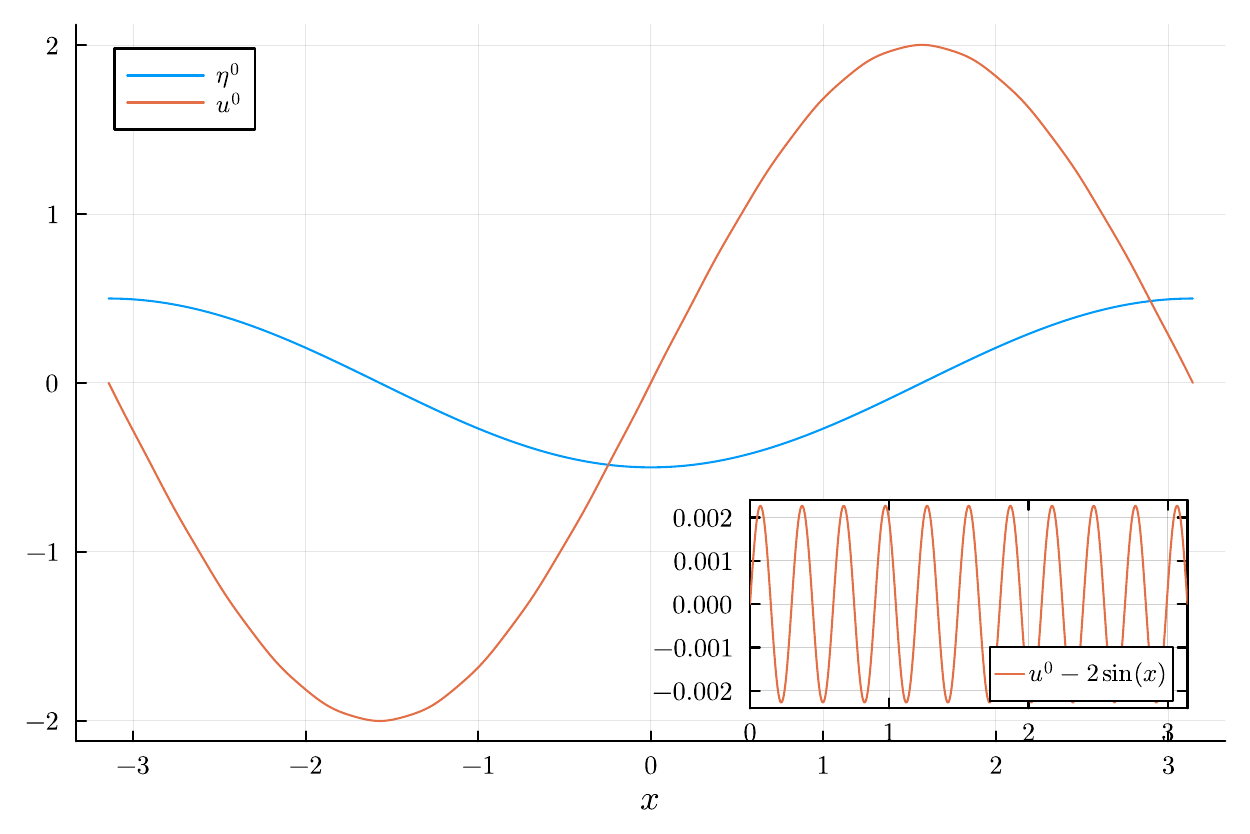}}\qquad
	\subfloat[Plot illustrating the convergence of the numerical schemes~\eqref{eq.SV-smooth} and~\eqref{eq.SV-sharp} as the number of collocation points \(2M\) increases. The plot shows the relative error of the numerical solution measured in the \(L^2\)-norm, \(E_0\) and in the \(H^1\)-norm, \(E_1\) for \(2M = 2^j, j= 6,\ldots, 14\) when using either sharp or smooth low-pass filters. To illustrate, the blue and orange lines have slopes \(-2\) and \(-1\) respectively. The numerical scheme exhibits spectral convergence with both sharp and smooth low-pass filters.]{\label{fig.num2}\includegraphics[width=0.45\linewidth]{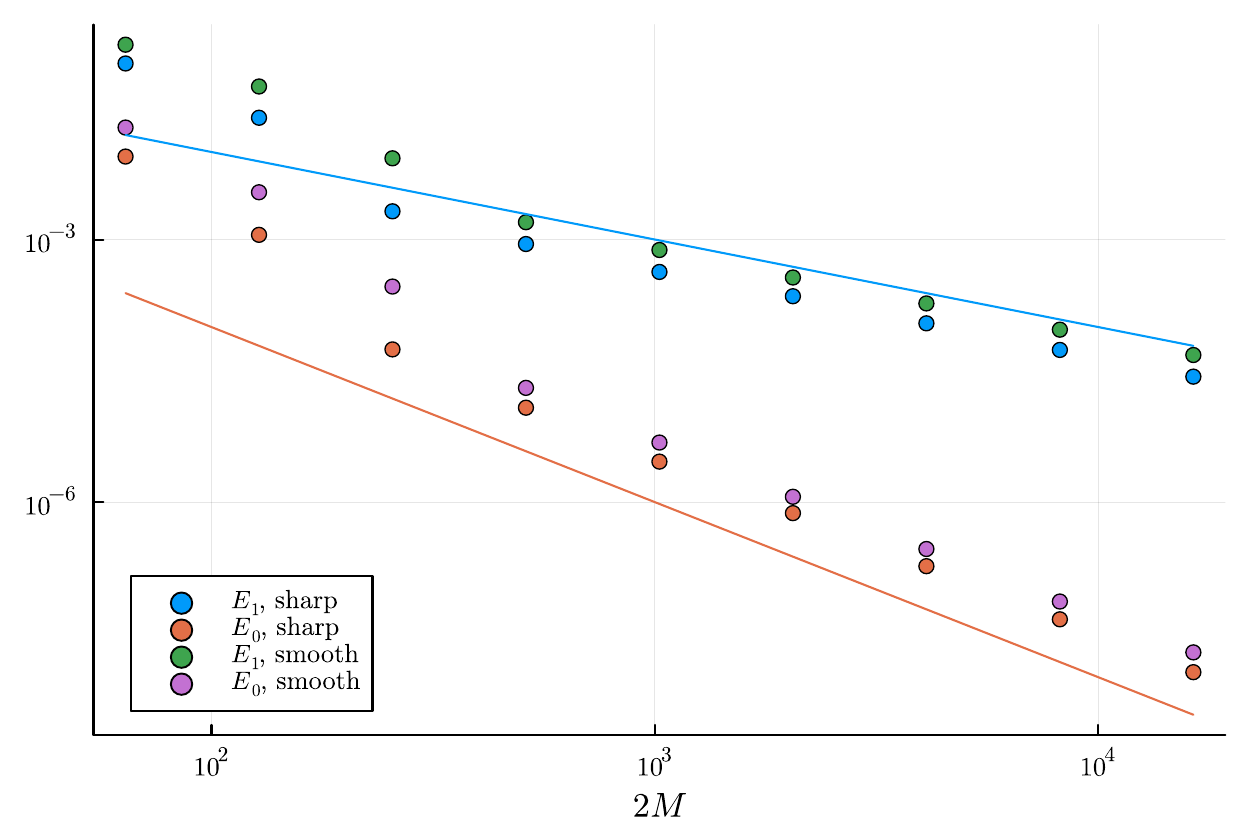}}\\
	\caption{Experiments with initial data~\eqref{eq.init2}.}
\end{figure}

Interestingly, we do observe a difference between sharp and smooth low-pass filters when (barely) violating the non-cavitation assumption \(1 + \eta^{0} >0\). Let 
\begin{equation}\label{eq.init_zero_depth}
	\eta_{0,N}(x) = -\cos(x), \quad u_{0,N}(x) = \sin(x) + \frac{\sin(Nx)}{N^2}.
\end{equation}
Figure~\ref{fig.zerodepth} shows the second derivative of the velocity at time \(T= 0.1\) for smooth and sharp low-pass filters for \(2M = 2^{10}\) and \(2M = 2^{12}\). The second derivative is uniformly bounded when using the smooth low-pass filter whereas it is not around the point \(x=0\) where the non-cavitation assumption is violated when using the sharp low-pass filter.

\begin{figure}[tbp]
	\centering
	\subfloat[Plot of the second derivative of the velocity \(u_N\) at time \(T=0.1\) with initial data~\eqref{eq.init_zero_depth} with \(2M = 2^{10}\). The orange and blue lines show the solutions found by using the the numerical schemes~\eqref{eq.SV-smooth} and~\eqref{eq.SV-sharp} with smooth and sharp low-pass filters respectively.]{\label{fig:zd10}\includegraphics[width=0.45\linewidth]{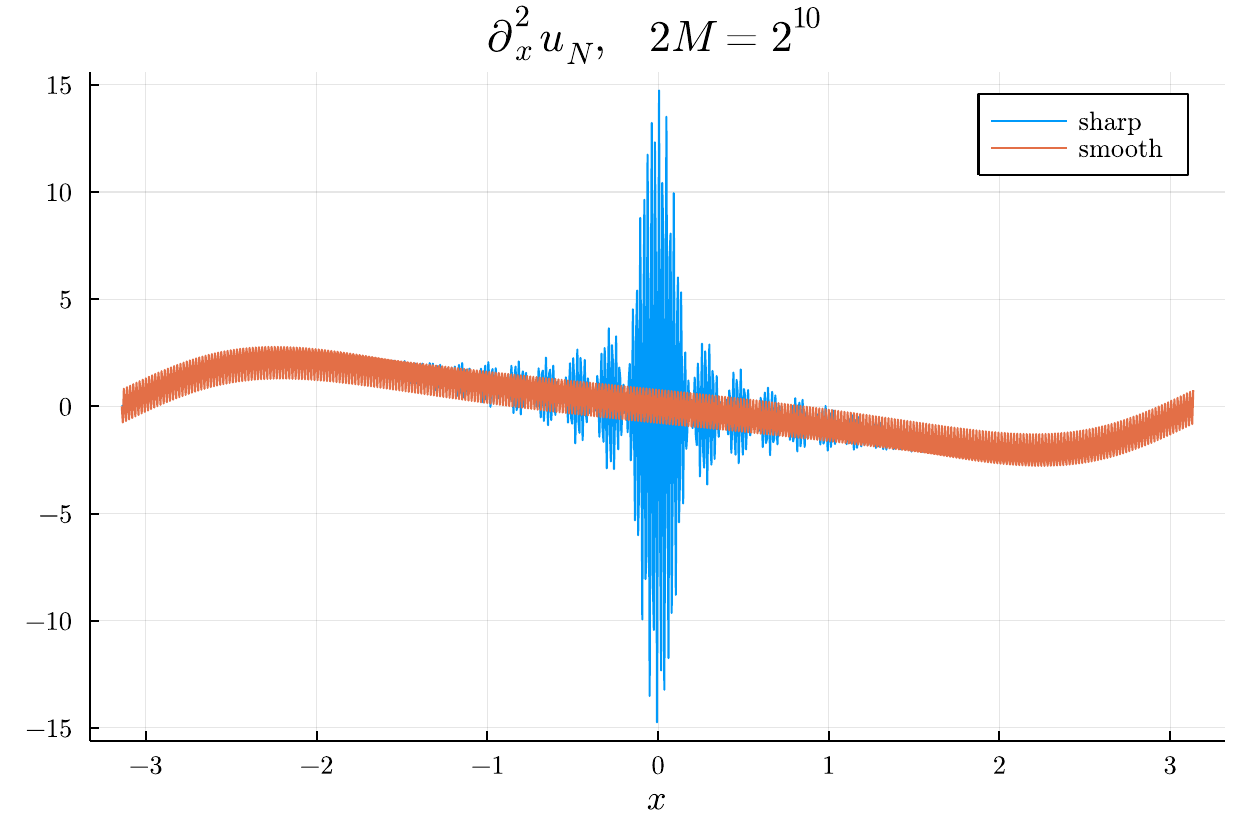}}\qquad
	\subfloat[Plot of the second derivative of the velocity \(u_N\) at time \(T=0.1\) with initial data~\eqref{eq.init_zero_depth} with \(2M = 2^{12}\). The orange and blue lines show the solutions found by using the numerical schemes~\eqref{eq.SV-smooth} and~\eqref{eq.SV-sharp} with smooth and sharp low-pass filters respectively.]{\label{fig:zd12}\includegraphics[width=0.45\linewidth]{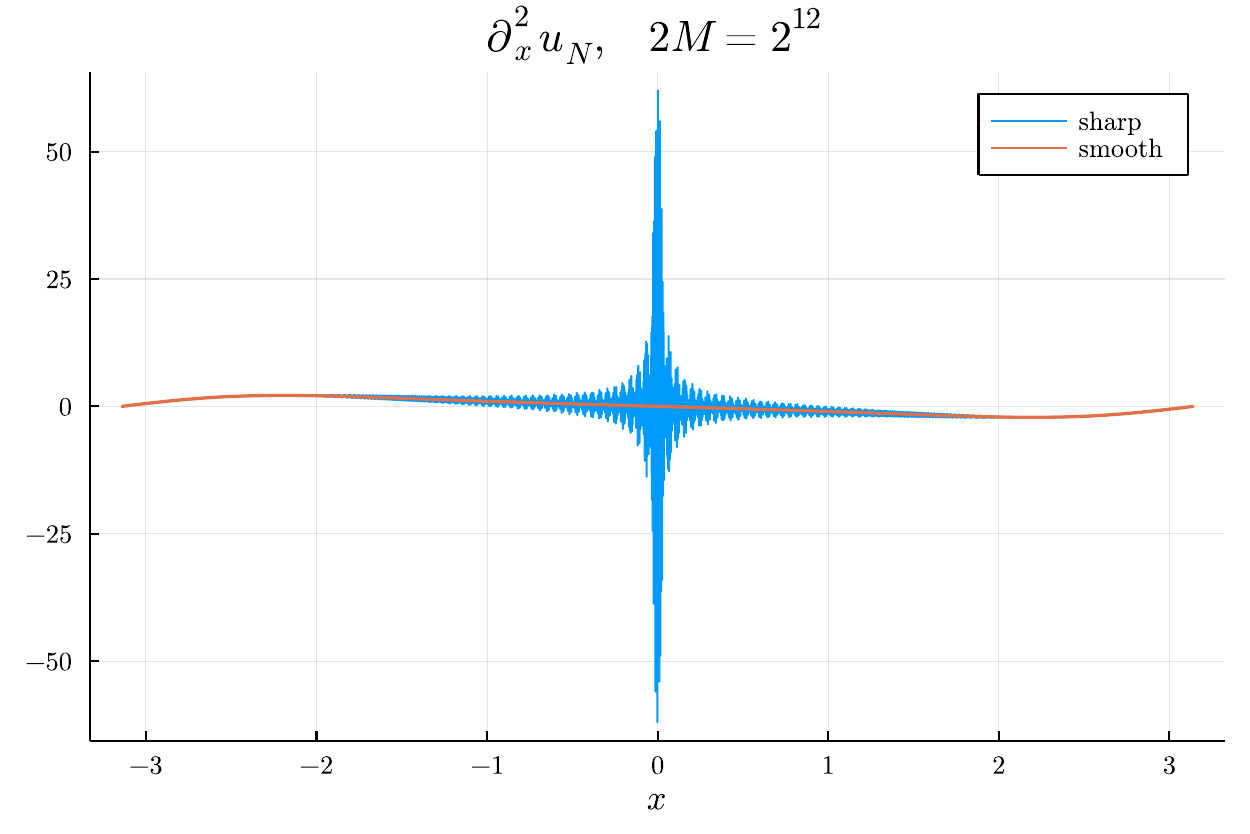}}\\
	\caption{Experiments with initial data~\eqref{eq.init_zero_depth}.}
	\label{fig.zerodepth}
\end{figure}

\subsection{Numerical experiments in dimension two}\label{sec.SV-num-2}
The numerical simulations of the Saint-Venant system with \(d=2\) are analogous to the case when \(d=1\) and have been executed using the same Julia package WaterWaves1D~\cite{DucheneNavaro}. They are also reproducible using the scripts available at \url{WaterWaves1D.jl/examples/StudySaintVenant.jl}. We let \(2M\) denote the number of collocation points in each of the two spatial dimensions, \(x\) and \(y\), which form a grid with \(4M^{2}\) collocation points. As in the previous section, we set \(N = \lfloor 2M/3 \rfloor\) and let \(\P_N, \S_N\) be as described in the introduction. In particular, notice \(\S_N\) is now a composition of one-dimensional low-pass filters: \(\S_N \coloneq \Diag(S_N(D))\) with \(S_N((k_1,k_2)) \coloneq S(k_1/N) S(k_2/N)\), where we set \(S(\cdot) \coloneq \max\big(0,\min\big(1,2-2|\cdot|\big)\big)^2\). We denote numerical approximations by \(\eta_N, u_N, v_N\). The values at collocation points \(\bm\eta, \bu, \bv\) and associated discrete Fourier modes \(\bm\hat{\bm\eta}, \hat\bu, \hat\bv\) are $2M$-by-$2M$ matrices.

Our main interest in studying numerically the systems in two dimensions is to examine whether we observe any difference between the Hamiltonian and non-Hamiltonian version of the Saint-Venant system, respectively~\eqref{eq.SV-2} and~\eqref{eq.SV}. While there is indeed a difference with respect to stability of the numerical schemes ---we observe instabilities with the sharp low-pass filter when violating \(1 + \eta^{0} - (u^{0})^2 - (v^{0})^2>0\) for~\eqref{eq.SV-2}, but not for~\eqref{eq.SV}--- this can be explained by the difference between the hyperbolicity domains $\cU_\sH$ and $\cU$ rather than by the presence or absence of a Hamiltonian structure. 
Because the setting of dimension $d=2$ is computationally costlier, we use timestep \(dt = 5\times 10^{-4}\) in our numerical experiments and, when calculating relative errors, we take as a reference solution the numerical solution computed with \(2M = 2^{10}\). 

We test the numerical method on initial data in \(H^s((2\pi\TT)^2)^3\) of the form 
\begin{equation}\label{eq.init2D}
	\begin{split}
		\eta^{0}(x,y) &= (h_0 - 1)\cos(x)\cos(y),\\
		u^{0}(x,y) &= u_{\l} \sin(x)\cos(y) + u_{\h}\frac{\sin(Nx)\cos(Ny)}{N^s},\\
		v^{0}(x,y) &= v_{\l} \cos(x)\sin(y) + v_{\h}\frac{\cos(Nx)\sin(Ny)}{N^s},
	\end{split}
\end{equation}
where \(h_0>0, u_{\l}, v_{\l}, s\geq0\) are real numbers. When \(u_{\l}= v_{\l}\) and \(u_{\h}= v_{\h}\), the initial data is irrotational, and systems~\eqref{eq.SV} and~\eqref{eq.SV-2} are equivalent.

For the standard, non-Hamiltonian Saint-Venant system~\eqref{eq.SV} our numerical results when \(d=2\) align with the numerical results when \(d=1\). That is, the numerical approximation converges with order \(s\) for tested initial data in \(H^s(2\pi\TT)\) as long as \(1+\eta^{0}>0\) when using both the smooth and sharp low-pass filters. Figure~\ref{fig.2DnonHam_conv} shows a log-log plot of the relative error measured in the \(L^2\)-norm, \(E_0(\bU_N)\) and in the \(H^1\)-norm, \(E_1(\bU_N)\), at time \(T= 0.1\) for \(2M = 2^j, j=5,\ldots,9\) and \(N= \lfloor 2M/3 \rfloor\) with initial data~\eqref{eq.init2D} with \(h_0=0.5, u_{\l} =-v_{\l} = 0.5, u_{\h} = -v_{\h} = 1\) and \(s=2\). Analogous results also hold for other values of \(s\) (we have tested \(s=2.5, 3\)). 
\begin{figure}[tbp]
	\centering
	\subfloat[System~\eqref{eq.SV}]{
		\label{fig.2DnonHam_conv}\includegraphics[width=0.5\linewidth]{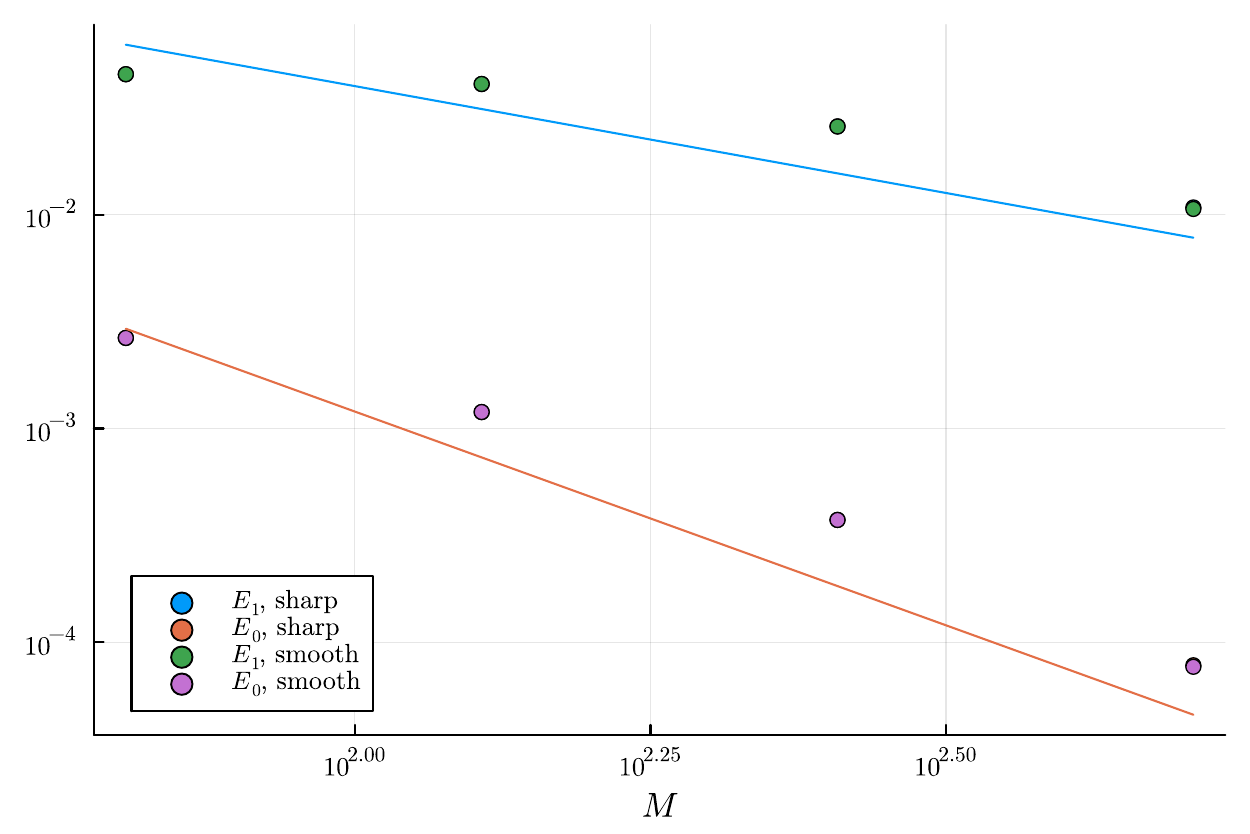}}
	\subfloat[System~\eqref{eq.SV-2}]{	\label{fig.2DHam_conv}\includegraphics[width=0.5\linewidth]{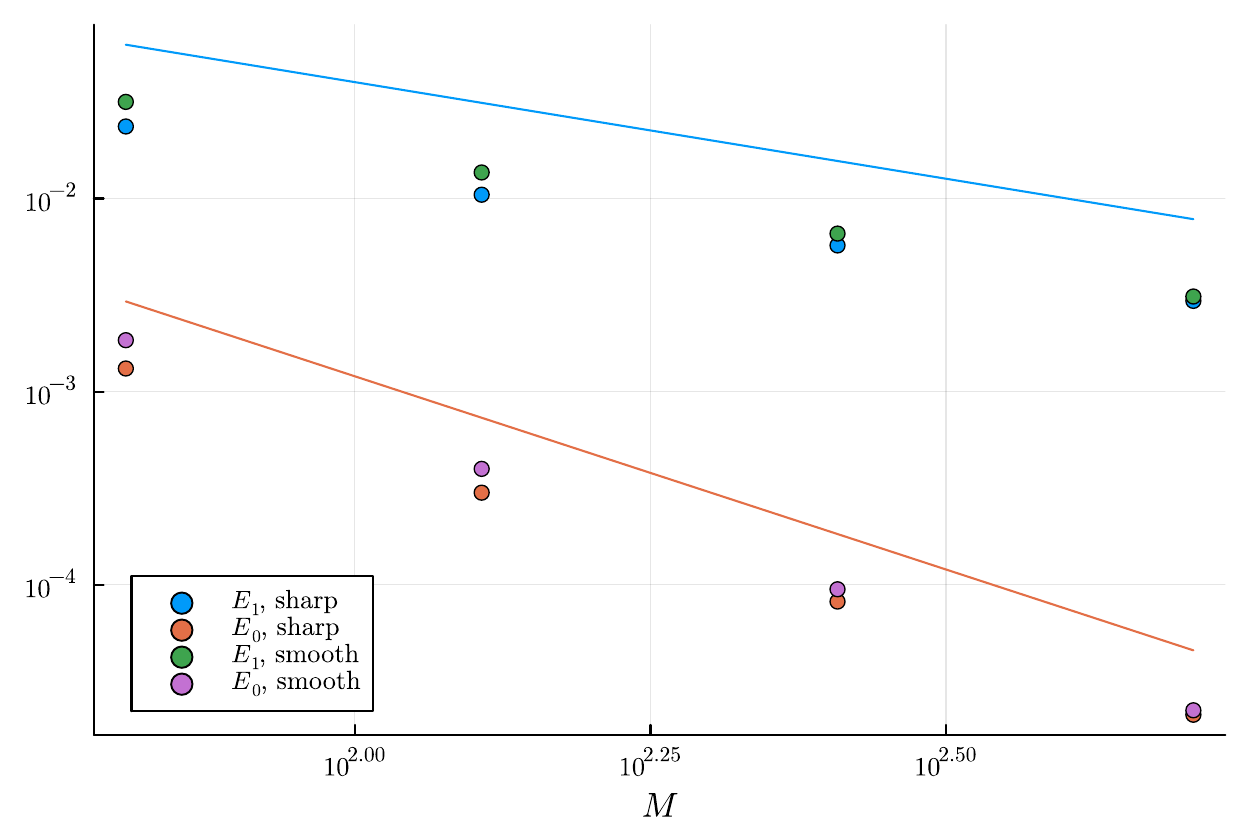}}
	\caption{Plot illustrating the convergence of the numerical schemes~\eqref{eq.SV-sharp} and~\eqref{eq.SV-smooth} for the systems~\eqref{eq.SV} (in the left) and~\eqref{eq.SV-2} (in the right) in two spatial dimensions as the number of collocation points \(2M\) increases. The plot shows the relative error of the numerical solution for initial data~\eqref{eq.init2D} with \(h_0=0.5, u_{\l} =- v_{\l} = 0.5, u_{\h} = -v_{\h} = 1\) and \(s=2\) at time \(T=0.1\). The initial data is in in \(H^2((2\pi\TT)^2)^3\) and the relative error is measured in the \(L^2\)-norm, \(E_0\) and in the \(H^1\)-norm, \(E_1\) for \(2M = 2^j, j= 6,\ldots, 9\) when using either sharp or smooth low-pass filters. To illustrate, the blue and orange lines have slopes \(-2\) and \(-1\) respectively. The numerical scheme exhibits spectral convergence with both sharp and smooth low-pass filters, for both systems.}
\end{figure}

We have not observed any instabilities in the numerical approximation of system~\eqref{eq.SV} due to the use of sharp low-pass filter  for any of the initial data we tested satisfying \(1+\eta^{0}>0\). Just as in the case of dimension one, we observe instabilities when violating the non-cavitation assumption \(1+\eta^{0}>0\) for the sharp low-pass filter, but not the smooth low-pass filter, see Figure~\ref{fig.2D_zerodepth}.

\begin{figure}[tbp]
	\centering
	\subfloat[Solution computed with the sharp low-pass filter and \({2M = 2^{9}}\).]{\label{fig:2Dzd9_sharp}\includegraphics[width=0.45\linewidth]{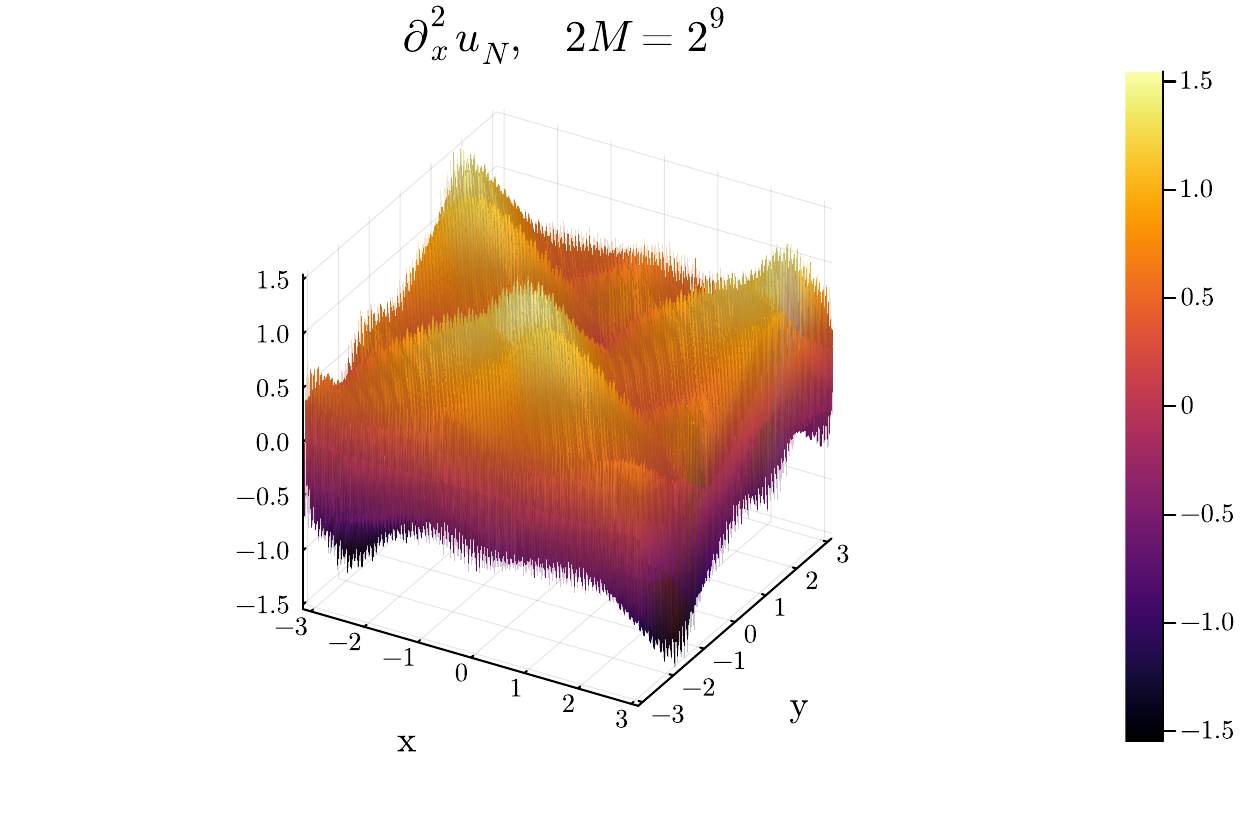}}\qquad
	\subfloat[Solution computed with the smooth low-pass filter and \({2M = 2^{9}}\).]{\label{fig:2Dzd9_smooth}\includegraphics[width=0.45\linewidth]{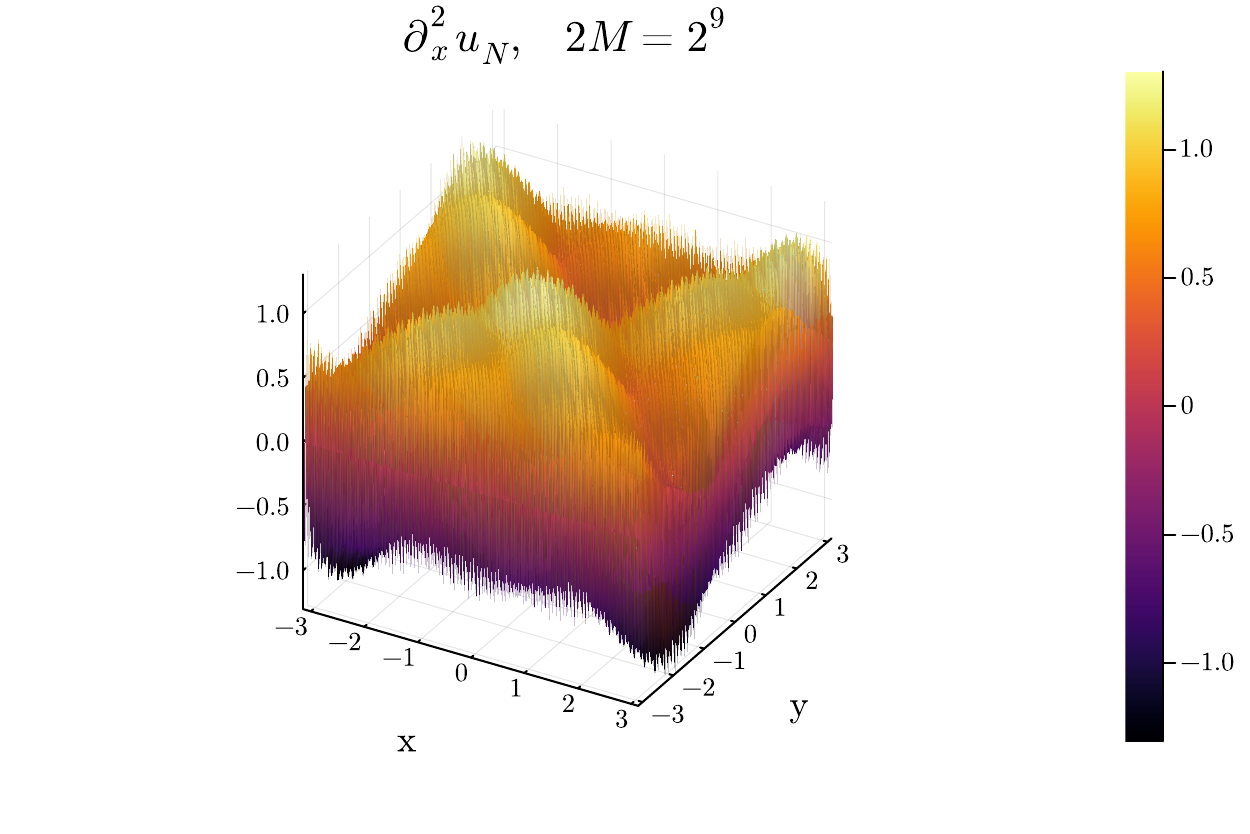}}\\
	\subfloat[Solution computed with the sharp low-pass filter and \({2M = 2^{10}}\).]{\label{fig:2Dzd10_sharp}\includegraphics[width=0.45\linewidth]{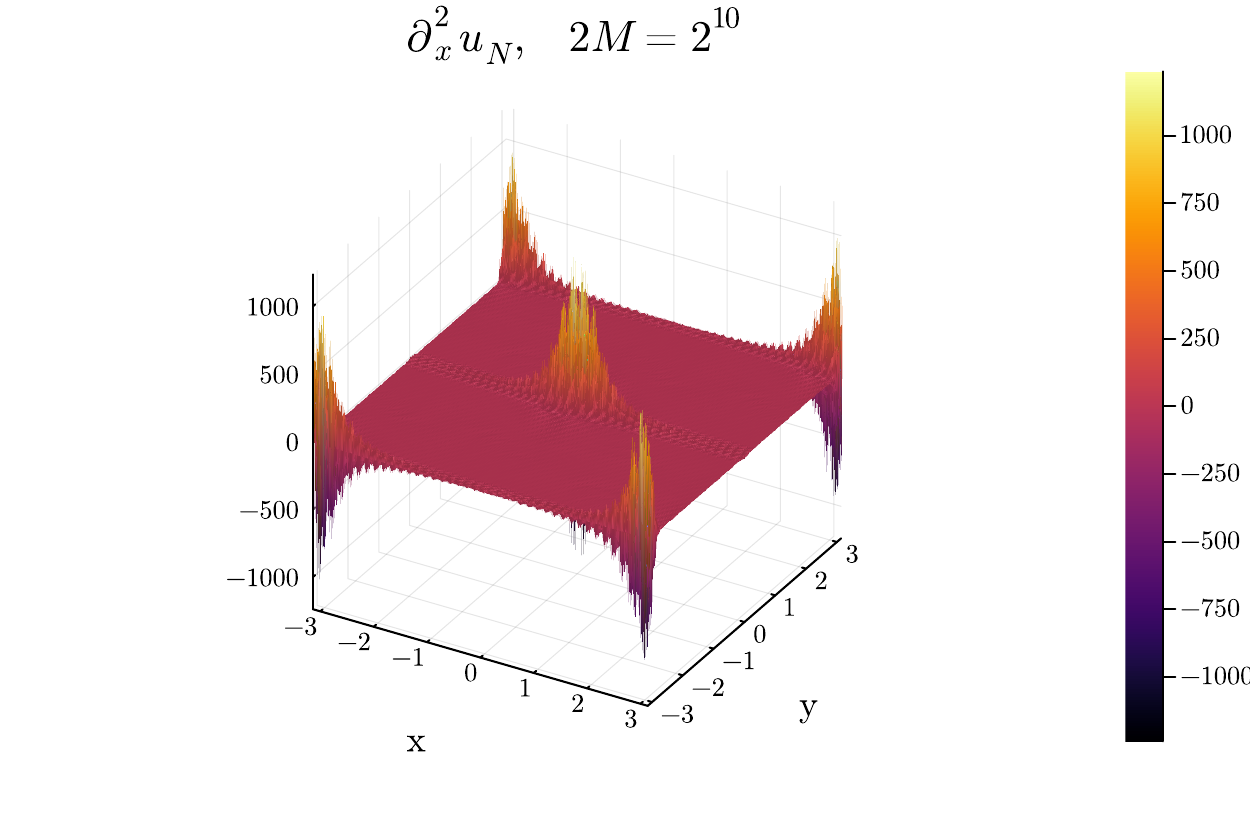}}\qquad
	\subfloat[Solution computed with the smooth low-pass filter and \({2M = 2^{10}}\).]{\label{fig:2Dzd10_smooth}\includegraphics[width=0.45\linewidth]{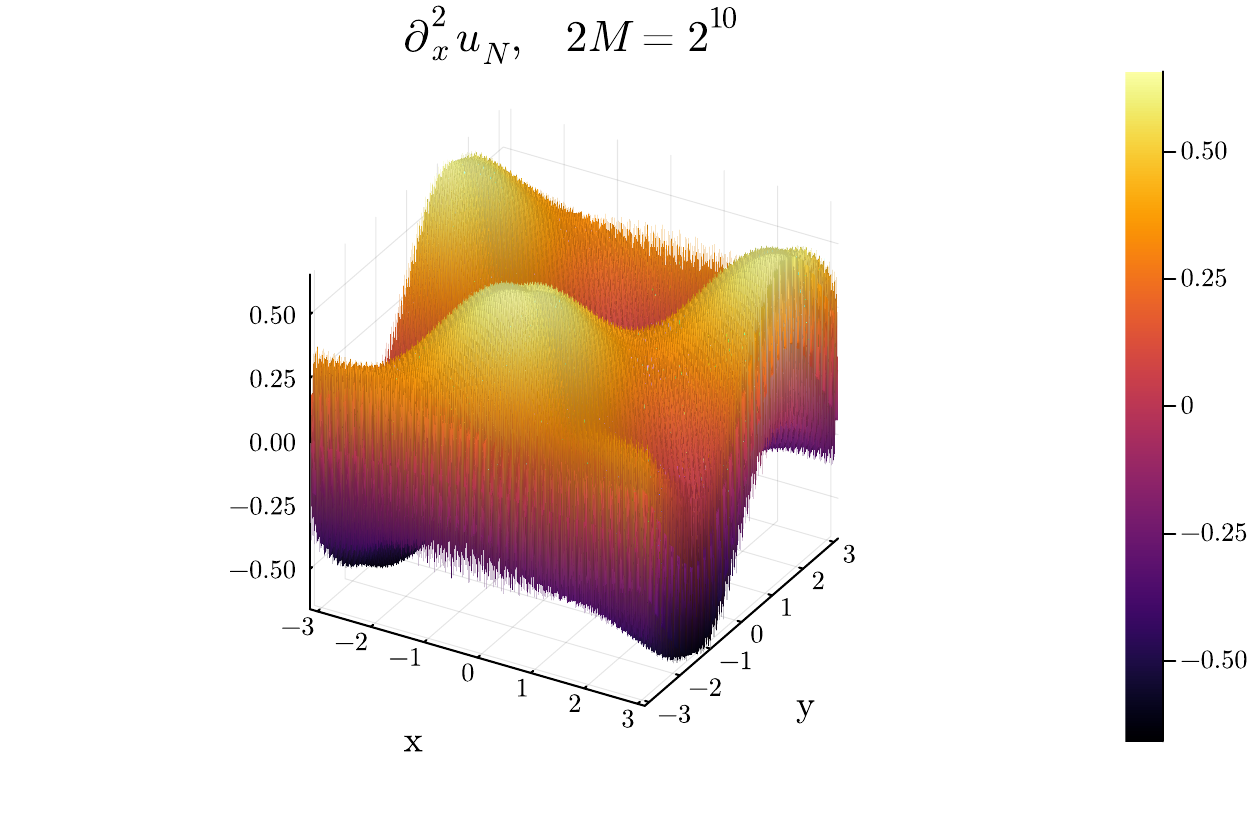}}\\
	\caption{Plots of the second derivative \(\partial_x^2 u_N\) of the numerical solution to~\eqref{eq.SV} for \(d=2\) at time \(T=0.1\), computed with either the sharp or smooth low-pass filter. The initial data is~\eqref{eq.init2D} with \(s=2, u_{\l} =- v_{\l} = 0.5, u_{\h} = -v_{\h} = 1\), negative minimial depth \(h_0=-0.1\) and \(N = \lfloor 2 M/3\rfloor\) for \(2M= 2^9\) or \(2M=2^{10}\).}
	\label{fig.2D_zerodepth}
\end{figure}

For the Hamiltonian Saint-Venant system~\eqref{eq.SV-2}, the numerical results are in line with the analysis. Whenever the hyperbolicity condition \(1 + \eta^{0} - (u^{0})^2 - (v^{0})^2>0\) is satisfied, the numerical scheme converges with the expected rate for both the smooth and the sharp low-pass filter. This is illustrated in Figure~\ref{fig.2DHam_conv} for initial data~\eqref{eq.init2D} with \(h_0, u_{\l} = - v_{\l} = 0.5, u_{\h} = -v_{\h} = 1\) and \(s=2\).

We observe instabilities in the numerical approximation of system~\eqref{eq.SV-2} when using the sharp low-pass filter  for initial data violating \(1 + \eta^{0} - (u^{0})^2 - (v^{0})^2>0\) but not  \(1 + \eta^{0} >0\). This is shown in Figure~\ref{fig.2D_strict_hyperbolicity}. There, we take as initial data~\eqref{eq.init2D} with \(h_0=0.5, u_{\l}=- v_{\l}=2, u_{\h} =  -v_{\h} = 1\) and \(s = 2\).
We relate these instabilities to the lack of well-posedness of the underlying system~\eqref{eq.SV-2} when the hyperbolicity condition \(1 + \eta^{0} - (u^{0})^2 - (v^{0})^2>0\) fails; see Lemma~\ref{lem.SV}. 
The numerical scheme with the smooth low-pass filter does not exhibit instabilities for the tested values of \(M\), but limited computational power prevents us from testing very large values of \(M\) in dimension two. 

\begin{figure}[tbp]
	\centering
	\subfloat[Solution computed with the sharp low-pass filter and \(2M = 2^{9}\).]{\label{fig:2Dsh9_sharp}\includegraphics[width=0.45\linewidth]{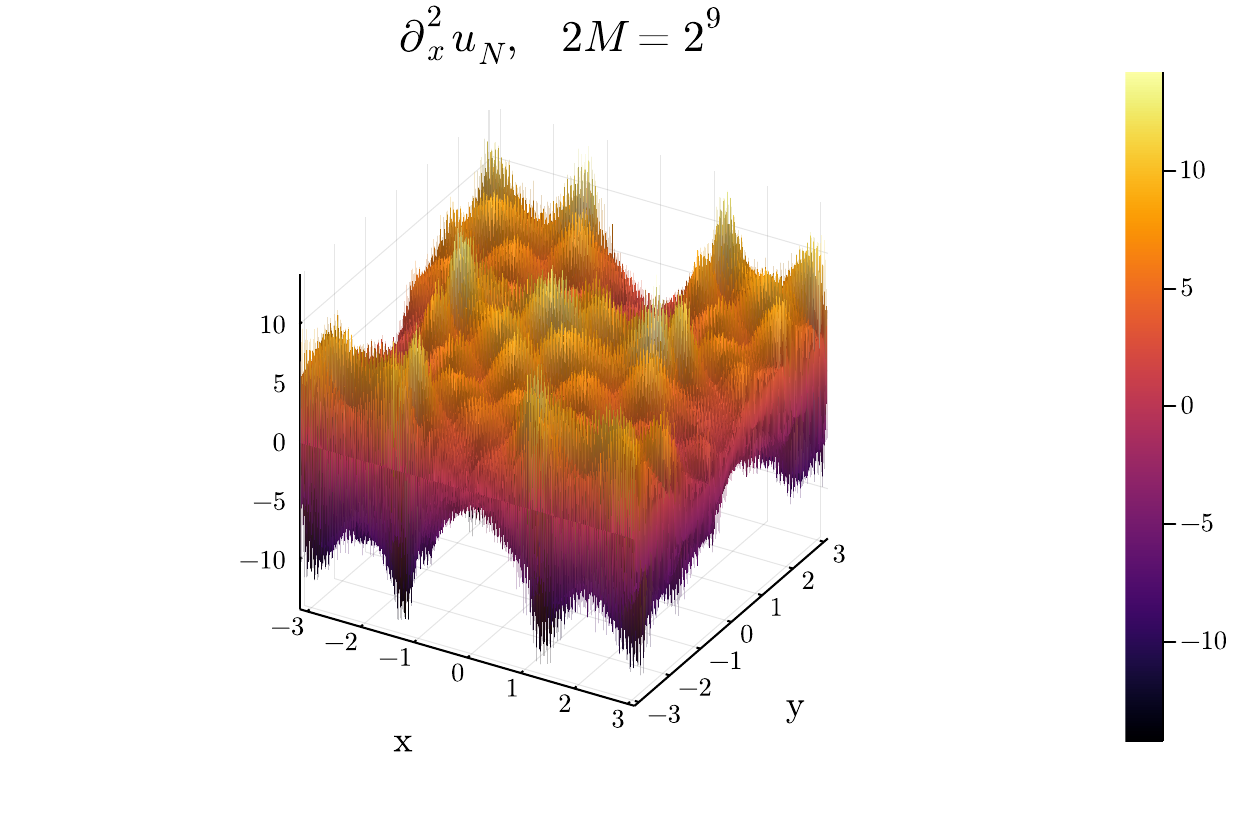}}\qquad
	\subfloat[Solution computed with the smooth low-pass filter and \(2M = 2^{9}\).]{\label{fig:2Dsh9_smooth}\includegraphics[width=0.45\linewidth]{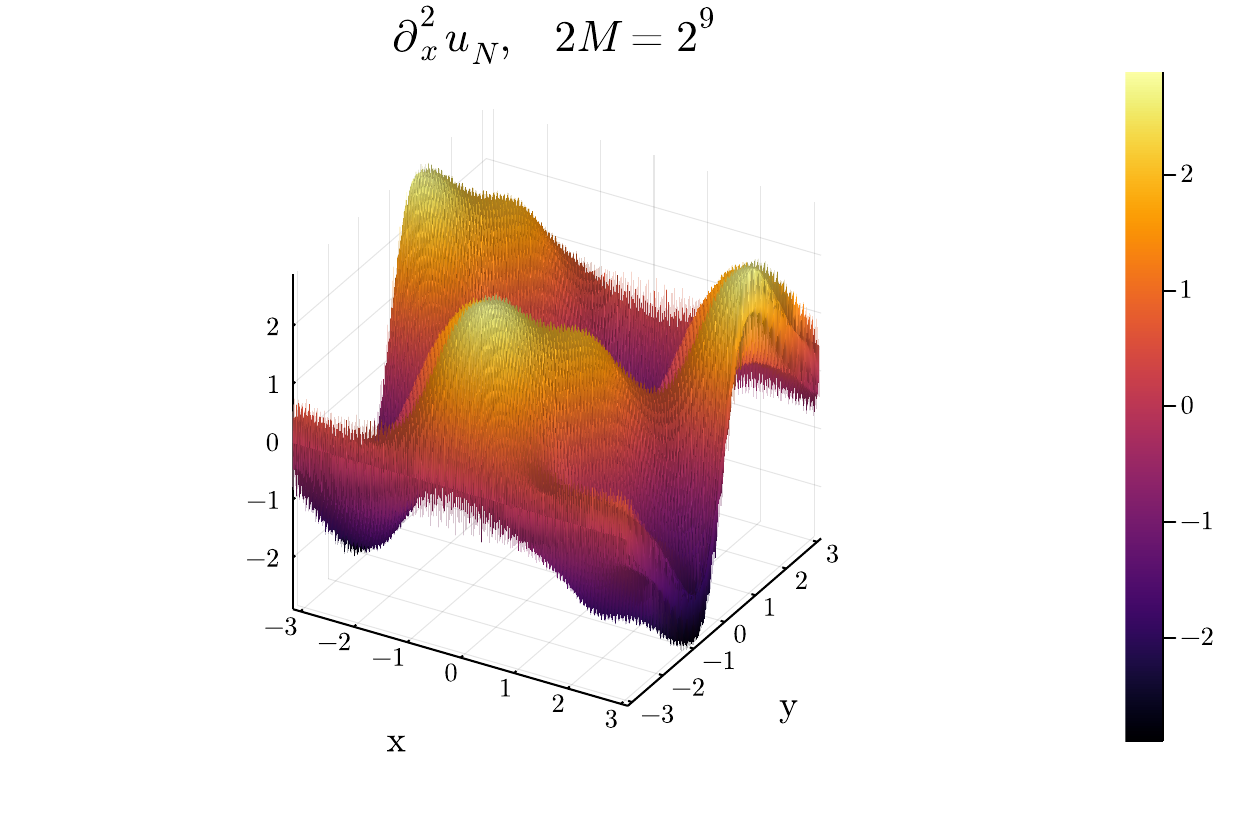}}\\
	\subfloat[Solution computed with the sharp low-pass filter and \(2M = 2^{10}\).]{\label{fig:2Dsh10_sharp}\includegraphics[width=0.45\linewidth]{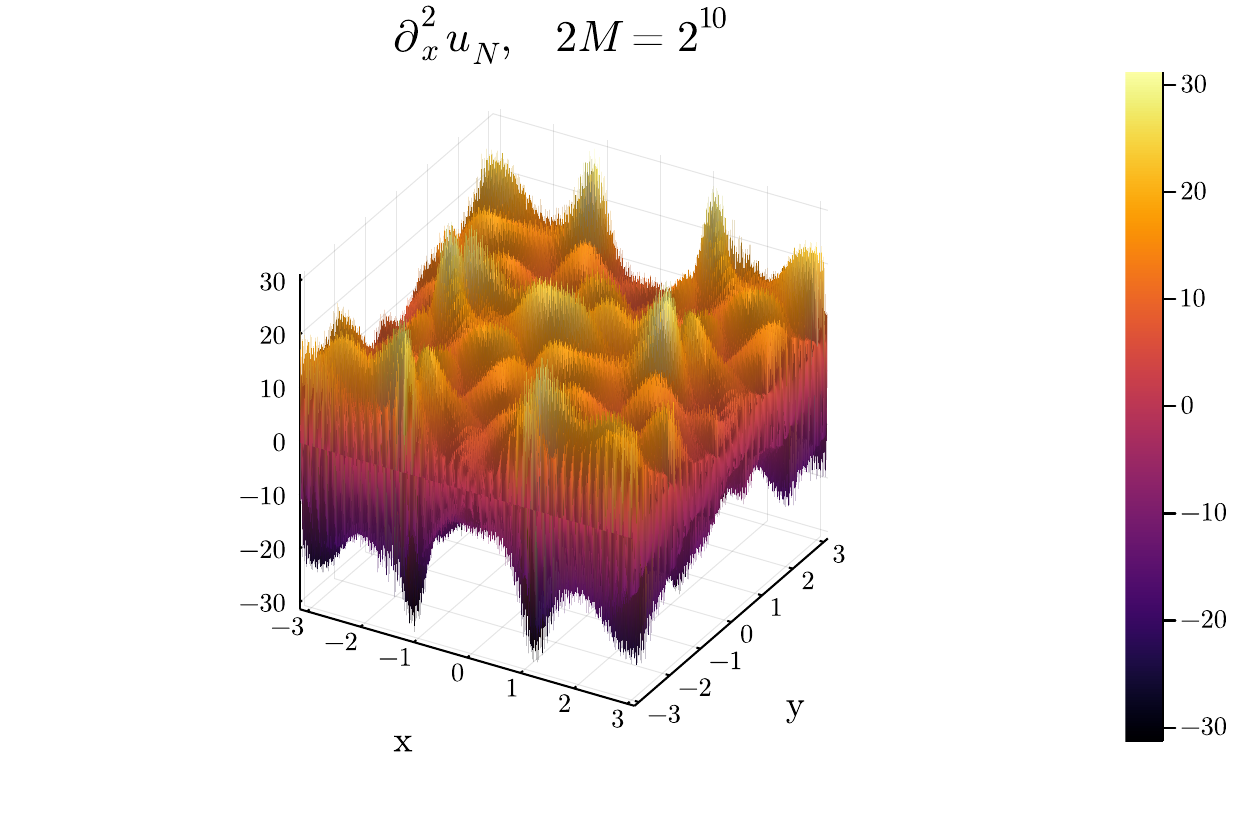}}\qquad
	\subfloat[Solution computed with the smooth low-pass filter and \(2M = 2^{10}\).]{\label{fig:2Dsh10_smooth}\includegraphics[width=0.45\linewidth]{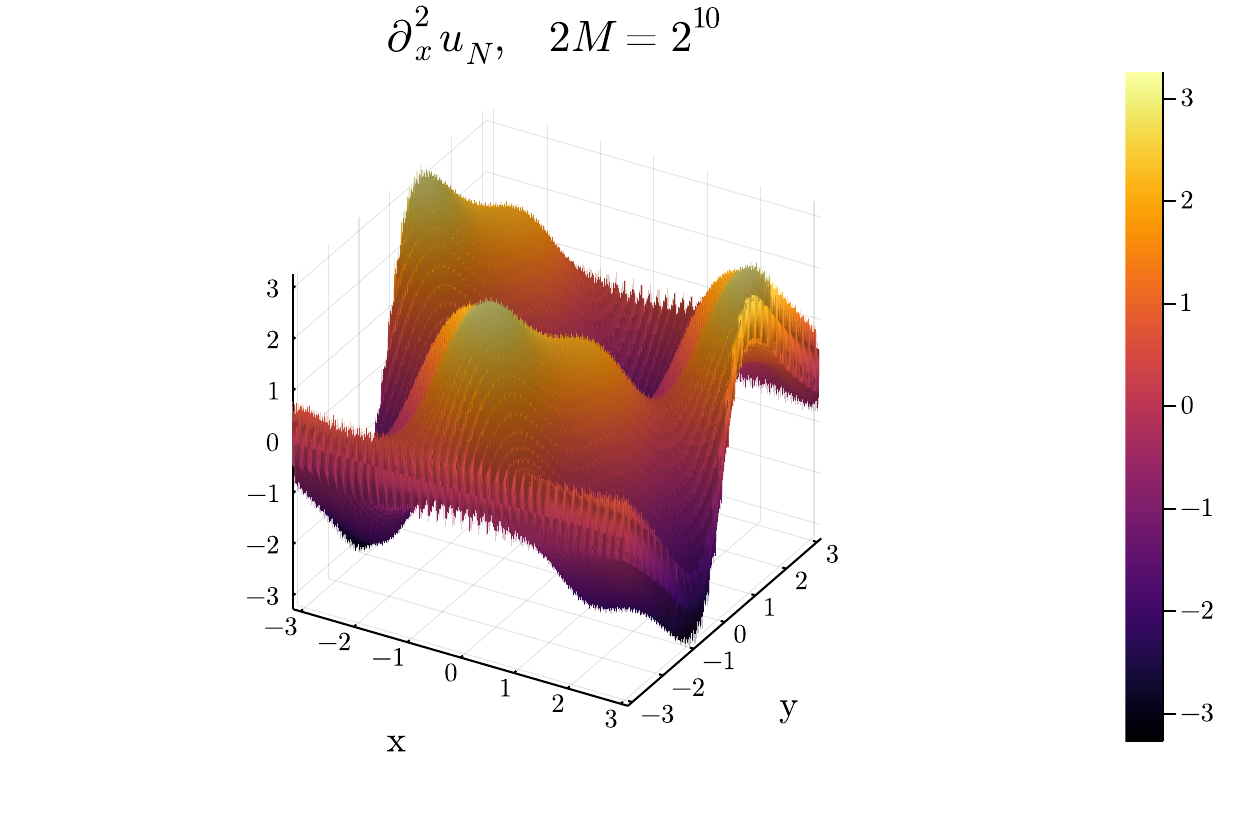}}\\
	\caption{Plots of the second derivative \(\partial_x^2 u_N\) of the numerical solution to~\eqref{eq.SV-2} for \(d=2\) at time \(T=0.1\), computed with either the sharp or smooth low-pass filter. The initial data is~\eqref{eq.init2D} with \(s=2, u_{\l} = - v_{\l} = 2, u_{\h} = -v_{\h} = 1\), positive minimal depth \(h_0=0.5\) and \(N = \lfloor 2M/3\rfloor\) for \(2M= 2^9\) or \(2M=2^{10}\).}
	\label{fig.2D_strict_hyperbolicity}
\end{figure}

	\appendix
	\section{Technical tools}
	The following results are standard, and proofs in the Euclidean space ({\em e.g.}~\cite[Theorem 8.3.1]{Hormander97} for product estimates) straightforwardly adapt to the periodic setting.
	
	\begin{Proposition}[Continuous embedding]\label{prop.embedding}
		Let $s\in\RR$ , $s>d/2$ and \(f\in H^{s}((2\pi\TT)^d)\). Then \(f\in L^\infty((2\pi\TT)^d)\) and 
		\[\big|f \big|_{L^\infty} \leq C(s) \big|f \big|_{H^{s}}.\]
	\end{Proposition}

	\begin{Proposition}[Interpolation inequality]\label{prop.interpolation_inequality}
		Let $s_1,s_2\in\RR$ and \(f\in H^{s_1}((2\pi\TT)^d)\cap H^{s_2}((2\pi\TT)^d)\). Then for any \(0\leq\theta\leq 1\), \(f\in H^{\theta s_1+(1-\theta)s_2}((2\pi\TT)^d)\) and 
		\[\big|f \big|_{H^{\theta s_1+(1-\theta)s_2}} \leq  \big|f \big|_{H^{s_1}}^\theta \big|f \big|_{H^{s_2}}^{1-\theta}.\]
	\end{Proposition}

	\begin{Proposition}[Product estimates]\label{prop.product_estimates}
		Let \(s_0 >d/2, s\geq -s_0\) and \(f\in H^{s}((2\pi\TT)^{d})\bigcap H^{s_0}((2\pi\TT)^{d}), g\in H^{s}((2\pi\TT)^{d})\). Then \(fg \in H^{s}((2\pi\TT)^{d})\) and 
		\[\big|fg \big|_{H^{s}} \leq C(s_0,s) \Big(\big|f \big|_{H^{s_0}}\big|g \big|_{H^{s}}+\big|f \big|_{H^{s}}\big|g \big|_{H^{s_0}}\Big).\]
		If moreover $s\leq s_0$ then
		\[\big|fg \big|_{H^{s}} \leq C(s_0,s) \big|f \big|_{H^{s_0}}\big|g \big|_{H^{s}}.\]
	\end{Proposition}
	Assuming polynomial nonlinearities, the following proposition is a straightforward consequence of product estimates. Extending this result to general (smooth) functions $P$ requires an analysis that is outside of the scope of the present paper.
	\begin{Proposition}[Composition estimates]\label{prop.composition_estimates} 
		Let \(s_0 >d/2, s\geq -s_0\), \(f,g\in H^{s}((2\pi\TT)^{d})\bigcap H^{s_0}((2\pi\TT)^{d})\) and {$P\in\RR[X]$ a polynomial}. Then \(P(f),P(g) \in H^{s}((2\pi\TT)^{d})\) and 
		\begin{align*}\big|P(f) -P(g)\big|_{H^{s}} &\leq C(P,s_0,s,\big|f \big|_{H^{\max(s_0, s)}},\big|g \big|_{H^{\max(s_0, s)}}) \big|f -g\big|_{H^{s}},\\
		\big|P(f) -P(0)\big|_{H^{s}} &\leq C(P,s_0,s,\big|f \big|_{H^{s_0}}) \big|f \big|_{H^{s}}.
	\end{align*}
	\end{Proposition}

	\begin{Proposition}[Commutator estimates with symbols of order s]\label{prop.commutator_s}
		Let \(s_0>d/2, s\geq 0\) and \(\Lambda^s  = (\Id - \Delta)^{s/2}\). Let \(f\in H^{s}((2\pi\TT)^{d})\bigcap H^{s_0+1}((2\pi\TT)^{d}), g\in H^{s-1}((2\pi\TT)^{d})\bigcap H^{s_0}((2\pi\TT)^{d})\). Then 
		\[\big|[\Lambda^s ,f]g \big|_{L^2} \leq C(s_0, s)(\big|f \big|_{H^s}\big|g \big|_{H^{s_0}} + \big|f \big|_{H^{s_0 +1}}\big|g \big|_{H^{s-1}}).\]
	\end{Proposition}
	The following result is shown on \(\RR^d\) in~\cite[Lemma 4.5]{DucheneMelinand24}. The proof straightforwardly adapts to the periodic setting.
	\begin{Proposition}[Commutator estimates with operators of order zero]\label{prop.commutator_0}
		Let \(s_0 > d/2, s\geq0\) and $G(D)$ be a Fourier multiplier with symbol \(G\) satisfying \(\big|G \big|_{L^{\infty}}, \big|\left|\,\cdot\, \right| \nabla G \big|_{L^{\infty}} \leq C_G\). Let \(f\in H^{s_0+1}((2\pi\TT)^d)\bigcap H^{s}((2\pi\TT)^d)\), \(g \in H^{s-1}((2\pi\TT)^d)\). Then
	   \[\big|[G(D), f] g\big|_{H^s} \leq C(s_0, s)\, C_G\, \big|f \big|_{H^{\max(s_0+1,s)}}\big|g \big|_{H^{s-1}}.\]
   \end{Proposition}
   \begin{Remark}
   	Notice that for smooth symbols considered in this work, namely  $S_N(\cdot)=S(\cdot/N)$ where $S$ is even with
   	\[\begin{cases}
   		S(\bk)=1&\text{if $\max_{j=1,\ldots, d}|k_j|\leq 1/2$,}\\
   		S(\bk)=0&\text{if $\min_{j=1, \ldots, d}|k_j|\geq 1$,}\\
   		S(\bk)\in{[}0,1{]}&\text{otherwise,}
   	\end{cases}
   	\]
   	and $S^{1/2}$ is Lipschitz-continuous, $S_N^{1/2}$ satisfies the hypotheses of Proposition~\ref{prop.commutator_0} uniformly with respect to $N$. Indeed, by Rademacher's theorem we have that $S_N^{1/2}$ is differentiable almost everywhere and its derivative is essentially bounded, and since $S$ has compact support, $\left|\,\cdot\, \right| \nabla S_N^{1/2} \in L^{\infty}$. Moreover, we have
   	 \[\big|S_N^{1/2} \big|_{L^{\infty}}+ \big|\left|\,\cdot\, \right| \nabla S_N^{1/2} \big|_{L^{\infty}} = \big|S^{1/2} \big|_{L^{\infty}}+ \big|\left|\,\cdot\, \right| \nabla S^{1/2} \big|_{L^{\infty}}.\]
   \end{Remark}

   \section*{Acknowledgements} VD thanks Centre Henri Lebesque ANR-11-LABX-0020-0 for fostering an attractive mathematical environement. JUM acknowledges the support of the project IMod (Grant No. 325114) from the Research Council of Norway. 
   \medskip

	\bibliographystyle{abbrv}

\end{document}